\documentclass[10.9pt]{amsart}

\usepackage{amssymb,amsfonts,psfrag,amscd,enumerate,epsfig,latexsym,amsmath}
\usepackage{subfigure}
\usepackage{xypic}
\usepackage[colorlinks,urlcolor=green,citecolor=blue,linkcolor=blue]{hyperref}

\usepackage[letterpaper,top=3.0cm, bottom=2.7cm, left=3.0cm, right=3.0cm]{geometry}

\newtheorem{theorem}{Theorem}
\newtheorem{lemma}[theorem]{Lemma}
\newtheorem{corollary}[theorem]{Corollary}

\newtheorem{proposition}[theorem]{Proposition}

\theoremstyle{definition}

\newtheorem{definition}[theorem]{Definition}

\newtheorem{example}[theorem]{Example}

\newtheorem{algorithm}[theorem]{Algorithm}

\theoremstyle{remark}

\author{Valent\'{\i}n Mendoza}\address{Instituto de Matem\'atica e Computa\c c\~ao, Universidade Federal
de Itajub\'a,  Av. BPS 1303, Bairro Pinheirinho, CEP 37500-903, Itajub\'a, MG, Brazil.}
\email{valentin@unifei.edu.br}
\title{Dynamics forced by homoclinic orbits}


\date{October 30, 2015.}


\begin{document}
\openup 0.9\jot
\begin{abstract}
The complexity of a dynamical system exhibiting  a homoclinic orbit is given by 
the orbits that it forces. In this work we present a method, based in  
pruning theory, to determine the dynamical core of a homoclinic orbit of a Smale 
diffeomorphism on the $2$-disk. Due to Cantwell and Conlon, this set is uniquely 
determined in the isotopy class of the orbit, up a topological conjugacy, so 
it contains the dynamics forced by the homoclinic orbit. Moreover we apply the method 
for finding the orbits forced by certain infinite families of homoclinic horseshoe 
orbits and propose its generalization to an arbitrary Smale map.
\end{abstract}
\maketitle
\section{Introduction}

Since the Poincar\'e's discovery of homoclinic orbits, it is known that
dynamical systems with one of these orbits have a very complex behaviour.
Such a feature was explained by Smale in terms of his celebrated
horseshoe map \cite{Sma}; more precisely, if a surface diffeomorphism
$f$ has a homoclinic point then, there exists an invariant set
$\Lambda$ where a power of $f$ is conjugated to the \textit{shift} $\sigma$
defined on the compact space  $\Sigma_2=\{0,1\}^\mathbb{Z}$ of symbol sequences. 

To understand how complex is a diffeomorphism having a periodic or homoclinic 
orbit, we need the notion of forcing. First we will define it for 
periodic orbits.
\subsection{Braid types of periodic orbits}
Let $P$ and $Q$ be two periodic  orbits of homeomorphisms  $f$ and $g$ of the closed disk $D^2$, 
respectively. 
We say that $(P,f)$ and $(Q,g)$ are \textit{equivalent} if there is an 
orientation-preserving homeomorphism $h:D^2\rightarrow D^2$ 
with $h(P)=Q$ such that $f$ is isotopic to $h^{-1}\circ g\circ h$ relative to $P$. 
The equivalence class containing $(P,f)$ is called the \textit{braid type} 
$\operatorname{bt}(P,f)$. When the homeomorphism $f$ is fixed, it will be 
written $\operatorname{bt}(P)$ instead of $\operatorname{bt}(P,f)$.

Now we can define the forcing relation between two braid types $\beta$ and $\gamma$.
We say that  $\beta$ \textit{forces} $\gamma$, denoted by $\beta\geqslant_2\gamma$,
if every homeomorphism of $D^2$ which has an orbit with braid type $\beta$, exhibits also 
an orbit with braid type $\gamma$. So we say that $P$ \textit{forces} $Q$, 
denoted by $P\geqslant_2Q$, if $\operatorname{bt}(P)\geqslant_2 \operatorname{bt}(Q)$.
In \cite{Boy} Boyland  proved that $\geqslant_2$ is a partial order. 
If $\mathcal{E}(P)$ is the set of periodic points whose orbits are forced by $P$,
one can ensure that the dynamics of a diffeomorphism $f$ containing an 
orbit with the braid type of $P$ is at least as complicated as $f$ is restricted 
to $\mathcal{E}(P)$.  To find $\mathcal{E}(P)$ it is necessary to use the 
Nielsen-Thurston theory of classification of surface homeomorphisms up to isotopies. 
In fact, Asimov and Franks \cite{AsiFra} and Hall \cite{Hall1} showed
that if $f$  is pseudo-Anosov on $D^2\setminus P$ then $\mathcal{E}(P)$ is formed by the 
periodic points of its canonical representative $\phi$. If $P$ is of reducible type
then the Nielsen-Thurston representative $\phi$ does not have the minimal number 
of periodic orbits, but a refinement of $\phi$, due to Boyland  \cite{Boy,BoyIso},
is used to get a condensed map which satisfies that property.

An interesting case is when $P$ is a periodic orbit of a Smale map $f$ on $D^2$, 
that is, $f$ is an Axiom A map with a strong transversality between the invariant 
manifolds of its non-wandering points. In this context the  notion of \textit{bigon} 
has became a tool for finding the Nielsen-Thurston representative rel to $P$:
a \textit{bigon} is a simply  connected open region disjoint from $P$ which is bounded 
by a stable segment and an unstable segment. 
In fact, Bonatti and Jeandenans \cite{BonLanJea} have proved that if there are no
wandering bigons rel to a periodic orbit $P$ of a transitive basic piece $K$, 
there exists a pseudo-Anosov map $\phi$, with possibly 
$1$-pronged singularities,  and a semiconjugacy $\pi$ between $f$ and $\phi$ which 
is injective on the set of periodic orbits of $K$ except on the boundary ones. If we do 
not admit non-wandering bigons too, the Bonatti-Jeandenans map $\phi$ is actually 
pseudo-Anosov and then the non-boundary periodic orbits of $K$ are all forced by $P$. 
Thus the non-existence of bigons implies that the periodic orbits of $K$ are forced by $P$.
A similar result was obtained by Lewowicz and Ures in \cite{LewUre}: if $f$ is a Smale map 
without bigons relative to $P$, which is called \textit{exteriorly situated}, then its basic set 
is contained in the persistent set given by Handel \cite{Han} relative to $P$.
More precisely it follows from Grines  \cite{GriRep,GriTop} that if the non-wandering 
set of a Smale map contains an exteriorly situated  basic set that does not contain 
\textit{special pairs} of boundary periodic point (definition \ref{def:special}) then there exists a semiconjugacy between
 $f$ and a hyperbolic homeomorphism $f_0$ which can be considered the canonical 
 hyperbolic representative of the  homotopy class of $f$. See \cite{AraGri,GriMedPoc}.
 
Since there exists a decomposition in basic pieces for Smale maps, the same analysis can 
even be applied if $f$ has several transitive pieces and does not have bigons rel
 to $P$, considering the return maps to each basic piece. It will be the case, for example, 
 if $P$ is a renormalization of two or more orbits of  pseudo-Anosov type  \cite{MenOrd}.
\subsection{Forcing on homoclinic orbits}
Let us now suppose that $P$ is a homoclinic orbit to a fixed point of a 
Smale diffeomorphism. In that case there are substantial differences but also 
 useful similarities with the periodic case.  First in \cite{LosFor} Los has proved 
 that the forcing relation on periodic orbits can be extended to homoclinic or 
 heteroclinic orbits in a suitable topology. 
There are several methods for finding the dynamics forced by a
 homoclinic orbit of a diffeomorphism $f$. 
For example, in \cite{Coll1,Coll3}, Collins has constructed  
 surface hyperbolic diffeomorphisms associated to a homoclinic orbit which 
 can be used for approximating the  entropy of that orbit as 
 close as we want. His method uses the homoclinic tangle associated 
 to the orbit. In  \cite{TanYam1,TanYam2}, Yamaguchi and Tanikawa have studied 
the forcing relation of reversible homoclinic horseshoe orbits appearing in 
area-preserving H\'enon maps. In other direction Boyland and Hall have given
 conditions for which a periodic orbit is isotopy stable relative to a compact 
 set \cite{BoyHall}, and their result can be used for studying homoclinic orbits.

Since $f$ restricted to $M_P:=\operatorname{Int}(D^2)\setminus P$ is isotopic to an end periodic homeomorphism
(lemma \ref{lem:end}), 
it is more convenient to use the Handel-Miller theory of classification of
end periodic automorphisms on surfaces \cite{HanMil} as is presented in \cite{Fen}. 
In this case Cantwell and Conlon \cite{CanCon2} have proved that there exist  
a Handel-Miller map $h$ and a set $\mathcal{C}_P$, called \textit{dynamical core}, 
which is  the intersection of the pair of totally disconnected and transverse 
geodesic laminations, such that $h\colon\mathcal{C}\rightarrow\mathcal{C}$
 is uniquely determined by the isotopy class of $f$ on $M_P$. Thus one can say as definition that 
 $\mathcal{C}_P$ is the set of orbits 
forced by $P$ or that the \textit{dynamics of $h$ on  $\mathcal{C}_P$  is forced by $P$}: we say that 
an (finite or infinite) $h$-orbit $Q$ is \textit{forced} by $P$,  denoted by $P\geqslant_2 Q$,
  if  $Q\subset \mathcal{C}_P$. 
In the general case, if an end periodic map $f$ is irreducible \cite{Fen}, the union 
of the laminations fills $M$ and $f$ is isotopic to  a \textit{pseudo-Anosov}-like 
representative of the dynamics rel to $P$ which preserves a pair of geodesic laminations. 
In our case,  as in the work of Grines \cite{GriRep} for surfaces of finite topology, one has the following  result (theorem \ref{thm:first}).
\begin{theorem}\label{thm:forcing}
Let $f$ be a Smale map with a exteriorly situated basic set $K$  without special 
pairs  rel to an homoclinic orbit $P$. Then $\mathcal{C}_P=K$ up a 
topological finite-to-one semi-conjugacy $\iota:\mathcal{C}_P\rightarrow K$ which is injective on the set of 
non-boundary periodic points. 
\end{theorem} 
Therefore, in the hypothesis of theorem \ref{thm:forcing}, $K$ contains the orbits forced by $P$. 
\subsection{Results of the paper}
This is the context where our work is inserted. An important element of this paper
 is the  pruning theory introduced by de Carvalho in \cite{dC1}  which is a technique 
for eliminating dynamics of a surface homeomorphism. So in section \ref{background}
we state the differentiable pruning theorem (theorem \ref{theo:pruning}), due to 
de Carvalho and the author, 
that can be used to eliminate bigons of a Smale map and hence for finding the dynamical core
associate to a homoclinic orbit. This differentiable version of 
 pruning have been implemented for uncrossing invariant manifolds of a Smale
 diffeomorphism in a region called a \textit{pruning domain}. This point of 
 view was inspired by the work of P. Cvitanovi\'c \cite{Cvi} where a generic 
 one-folding map is interpreted as a \textit{partial horseshoe}, that is, a map 
 whose dynamics forbids or prunes certain horseshoe orbits. 

As an application we will study, in section \ref{sec:horseshoe}, the forcing 
relation of homoclinic orbits coming from the standard Smale  horseshoe $F$ stated in \cite{Sma}. 
In particular, it will be dealt homoclinic orbits to the fixed 
 point $0^\infty$, that is, orbits $P^w_0$ whose code in $\Sigma_2$ is ${}^\infty01{}^1_0w{}^1_0\cdot10^\infty$
  where $w$ is a finite word called \textit{decoration}. In \cite{dCarHallFor} de Carvalho and Hall 
  conjectured that the orbits forced  by $P^w_0$ are those ones that do not
   intersect  a region $\mathcal{P}_w$ called \textit{pruning region}, and that
    the forcing relation of periodic orbits depends basically of being able to determine it for 
homoclinic orbits. In this work we will conclude that the pruning method can be used for finding
 these pruning regions for \textit{certain} infinite families of decorations $w$. This follows
 proving that  after a finite number of applications of the differentiable pruning theorem one 
can eliminate all the  bigons rel those orbits. Thus $\mathcal{P}_w$ is precisely
the union of the pruning domains where the invariant manifolds were uncrossed.
 \begin{theorem}
 If $w$ lies to one of following classes of decorations: maximal, P-list or star, then there 
exists a pruning region $\mathcal{P}_w$ such that
$\mathcal{C}_{P^w_0}=\Sigma_2\setminus \bigcup_{i\in\mathbb{Z}}\sigma^{i}(\mathcal{P}_w)$
 up to topological finite-to-one semi-conjugacy which is injective on the set of 
non-boundary periodic points.
 \end{theorem}
In section \ref{sec:horseshoe} are defined all the concepts needed by theorem above.  
Thus the set of orbits that have to coexist with these homoclinic orbits is described completely.
 Among others results contained in the text, the pruning method allows us to prove a 
 Milnor-Thurston-like forcing on maximal homoclinic orbits (Corollary \ref{cor:maximal}): 
 \begin{corollary}
 If $w$ and $w'$ are maximal finite words satisfying that $w\geqslant_1w'$ and $\widehat{w}\geqslant_1\widehat{w'}$  then  $P^w_0\geqslant_2P^{w'}_0$, where $\geqslant_1$ denotes the 
 unimodal order and $\widehat{w}$ denotes the reverse word of $w$.
 \end{corollary}
A version of this result was proved by  Holmes and Whitley \cite{HolWhi} for homoclinic orbits that appear in strongly dissipative H\'enon maps.

In section \ref{sec:pruningbigons}  we introduce the \textit{pruning method}, that is an 
algorithm which could allow us to find, given a  homoclinic orbit $P$ of an arbitrary 
Smale map $f$, another Smale map $\psi$ without bigons  rel $P$. 
Hence the basic piece of such map $\psi$ has to contain all the dynamics 
forced by $P$. Unfortunately there is an inconvenient in our treatment: 
there is no guarantee that the method always stops in a finite number of steps; but even 
in that case there exists a pruning model, describing the dynamical core, which belongs 
to the isotopy class of a limit of hyperbolic pruning models \cite{MenOrd}.

 Now we would like to explain how the paper is organized. Section \ref{background}
 is devoted to relation between bigons and the forcing relation of 
 homoclinic orbits.  
  Section  \ref{sec:horseshoe} is devoted to the application of the differentiable 
	pruning theory to the study of  homoclinic horseshoe orbits and section \ref{sec:pruningbigons} will introduce the pruning method in a general form. 

\section{Dynamics forced by homoclinic orbits}\label{background}
Here we will define the notions that we are going to use. 
 We assume that the reader has some familiarity with Smale
 maps and pseudoanosov homeomorphisms. Good references
 for these topics are \cite{BonLanJea} and \cite{FatLauPoe}.
\subsection{Smale maps}
Let $f$ be a Smale map on the closed disk $D^2$ and suppose that $f$ has a 
unique non-trivial basic saddle set $K$ which is a totally disconnected
hyperbolic Cantor set. Suppose \textit{the domain of $K$}, 
that is, an invariant open region containing $K$ where the dynamics can be 
explained by the symbol dynamics of $K$ is $\Delta(K)=D^2\setminus\{s_1,s_2,\cdots,s_k\}$
where $s_i$ is a periodic point of $f$ for all $i=1,\cdots,k$. 
See the precise definition of $\Delta(K)$ in \cite{BonLanJea}. 
Thus the non-wandering set of $f$ is formed by $K$ and a finite set
of isolated saddles points, sinks and sources.
We will suppose that $f$ can be extended to $\partial D^2$ as the identity $f=Id$.

A saddle point $\bold{x}\in K$ is a \textit{$s$-boundary point} if $\bold{x}$ is a boundary 
point of $W^u(\bold{x},f)\cap K$, that is, if $\bold{x}$ is an accumulation point only
from one side by points in $W^u(\bold{x})\cap K$, or equivalently, if there exists 
a closed interval $I\subset W^u(\bold{x},f)$ having $\bold{x}$ as end-point such that 
$\operatorname{Int}(I)\cap K=\emptyset$. The set of $s$-boundary points is denoted by
$\partial_s K$. The \textit{$u$-boundary points} are defined similarly;
the set of the $u$-boundary points is denoted by $\partial_u K$. It is known \cite{NewPal} 
that there exists a finite number of periodic saddle points $p^s_1,\cdots,p^s_{n_s}$ and 
$p^u_1,\cdots,p^u_{n_u}$ such that $\partial_sK=\big(\cup_{i=1}^{n_s}W^s(p^s_i)\big)\cap K$ and 
$\partial_uK=\big(\cup_{i=1}^{n_u}W^u(p^u_i)\big)\cap K$.  
\begin{figure}[h]
\centering
\includegraphics[width=55mm,scale=0.6]{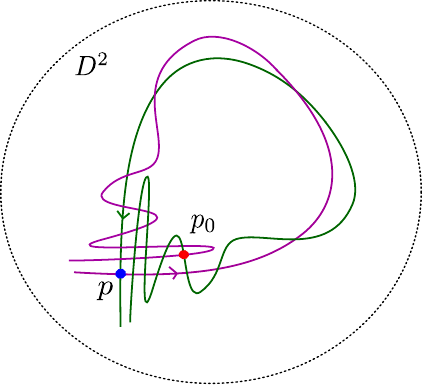}
\caption{A homoclinic orbit $P=\operatorname{Orb}(p_0)$ to a fixed point $p$.}
\label{fig1}
\end{figure}
We are going to study a homoclinic orbit 
$P=\{p_j\}_{j\in\mathbb{Z}}=\{f^j(p_0)\}_{j\in\mathbb{Z}}$ included in the 
intersection of the stable and unstable manifolds of a $s$- and $u$- boundary 
fixed point $p$, and let us suppose that the eigenvalues of $Df(p)$ are positive. 
Figure \ref{fig1} shows an example. The orbit of a point $\bold{x}\in K$ by $f$ will be 
denoted by $R=\operatorname{Orb}(\bold{x})=\{f^{i}(\bold{x})\}_{i\in\mathbb{Z}}$.

We need the following definition.
\begin{definition}\label{def:bigon} \normalfont
A \textit{bigon} is a simply connected open region $\mathcal{I}$ bounded by a segment of 
a stable manifold  $\theta_s\subset W^s(p_s)$ and a segment of an unstable manifold 
 $\theta_u\subset W^u(p_u)$, where $p_s$ and $p_u$ are saddle points of $K$.
\end{definition}
\begin{figure}[h]
\centering
\includegraphics[width=65mm,scale=0.6]{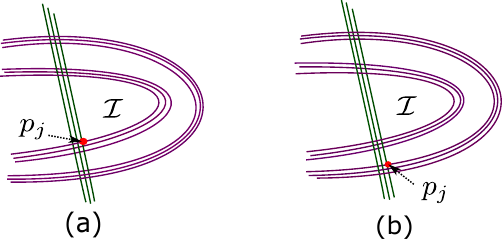}
\caption{Wandering bigons.}
\label{fig:bigonallowed}
\end{figure}
There are two types of bigons. A bigon $\mathcal{I}$ is called \textit{wandering} 
if it is disjoint from $K$. In this case $\partial \mathcal{I}=\theta_s\cup\theta_u$  
with $\theta_s\subset W^s(p_s)$ and $\theta_u\subset W^u(p_u)$, where $p_s$ and 
$p_u$ are boundary periodic points of $K$.  Figure \ref{fig:bigonallowed} shows
two wandering bigons.

The second type of bigons is the following: A bigon $\mathcal{I}$ is called \textit{non-wandering}
if  $\mathcal{I}\cap K\neq\emptyset$. In general, a non-wandering bigon contains
a wandering bigon is its interior. If it is not the case,  there are two possibilities:
$\mathcal{I}$ contains a $s$-boundary periodic saddle point $\bold{x}$ whose free branch of 
$W^u(\bold{x})$ belongs to the basin of an attracting periodic orbit, or  $\mathcal{I}$ contains 
a $u$-boundary saddle point $\bold{x}$ whose free stable manifold belongs to unstable set 
of a repelling periodic point. See figure \ref{fig:1prong}. 
\begin{figure}[h]
\centering
\includegraphics[width=95mm,scale=0.6]{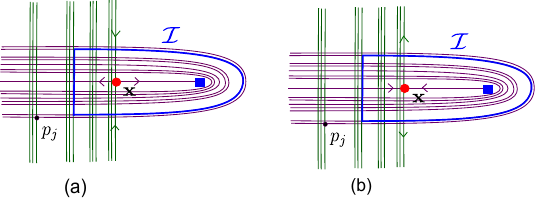}
\caption{Non-wandering bigons.}
\label{fig:1prong}
\end{figure}

If $f$ has no bigons relative to a homoclinic orbit $P$ then there are not 
non-wandering bigons and every bigon is as in figure \ref{fig:bigonallowed}(a) 
where $p_j$ represents an element of the homoclinic orbit, that is, there exists a $p_j\in P$ 
such that $\{p_j\}\subset \theta_s\cap\theta_u$. 
In this case we say that $K$ is \textit{exteriorly situated}  rel $P$.

If $P$ is a homoclinic orbit to a fixed point $p$, the following proposition 
proves that $f$ restricted to $K$ will be transitive providing that it does not 
have bigons.
\begin{proposition}\label{prop:transitive}
Let $P$ be a homoclinic orbit to a fixed point $p$ of a Smale map $f$ on $D^2$. 
If $f$ does not have bigons rel $P$ then $f$ is transitive on its basic set $K$.
\end{proposition}
\begin{proof}
Suppose that there exist at least two transitive disjoint basic sets 
$K_1$ and $K_2$ such that $\{p\}\cup P\subset K_1$. If $W^s(K_1)\cap W^u(K_2)\neq\emptyset$ 
then the elements $p_j$ of $P$ have to be situated in $W^s(K_1)\cap W^u(K_2)$, because
otherwise there would exist bigons rel to $P$. It is a contradiction since, in that case,
$\lim f^{-n}(p_j)$ goes to $p\in K_1\cap K_2$ when $n$ goes to $\infty$, 
which is clearly a contradiction. Hence $W^s(K_1)\cap W^u(K_2)=\emptyset$.
Similarly we can prove that $W^u(K_1)\cap W^s(K_2)=\emptyset$.  
Hence $K_1$ and $K_2$ are not homoclinically related. Thus one can suppose
that $P\subset K_1$ and $P\cap K_2=\emptyset$. Since $D^2$ is simply connected, it follows
that any homoclinic intersection happening in $K_2$ creates wandering and non-wandering
 bigons rel to $P$. It is a contradiction with the hypothesis. Hence $K_2=\emptyset$.
\end{proof}

\begin{definition}\label{def:special}
A pair of periodic points $\bold{x},\bold{y}\in K$ is \textit{$u$-special} if 
$W^u(\bold{x})\cup W^u(\bold{y})$ is accessible from inside boundary of a domain 
that is a continuous immersion of the open disk in $D^2$ and belongs to $D^2\setminus K$.
In an analogous manner, a pair of periodic points $\bold{x},\bold{y}\in K$ is 
\textit{$s$-special} if $W^s(\bold{x})\cup W^s(\bold{y})$ is accessible from inside boundary 
of a domain  that is a continuous immersion of the open disk in $D^2$ and belongs to $D^2\setminus K$.
\end{definition}

\begin{figure}[h]
\centering
\includegraphics[width=75mm,scale=0.7]{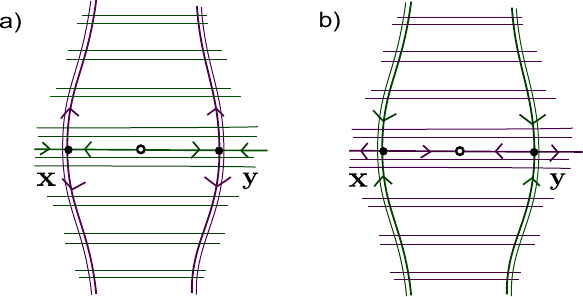}
\caption{In (a) is pictured a pair of $u$-special points and in (b) is pictured a pair 
of $s$-special points.}
\label{fig:special}
\end{figure}
See figure \ref{fig:special} for examples. 
\begin{proposition}
Let $K$ be an exteriorly situated basic set of a Smale map $f$ on $M_P$. Then $f$ is isotopic to 
a Smale map $f^*$ on a basic set $K^*$ by a semiconjugacy 
$\tau:K\rightarrow K^*$ such that 
\begin{itemize}
\item[(1)] $\tau(K)=K^*$, $\tau \circ f=f^*\circ\tau$;
\item[(2)] the set $Z\subset K^*$ of points $z$ whose preimage $\tau^{-1}(z)$ 
contains more than one point is such that $h^{-1}(Z)=K\cap (\cup W^u(p_i^u) \bigcup\cup W^s(p_j^s))$  where $\{p_i^u\}$ 
is the finite set of $u$-special periodic points and $\{p_j^s\}$ is the finite set of the $s$-special periodic points;
\item[(3)] $K^*$ does not have special pairs of points.
\end{itemize} 
\end{proposition}
\begin{proof}
Note that there exist a finite number of pair of periodic special points
 $\{\bold{x}_i,\bold{y}_i\}_{i=1}^k$. Suppose that $\bold{x}$ and $\bold{y}$ be a $u$-special 
 pair of boundary periodic points and let $B$ be the region between $W^u(\bold{x})$
 and $W^u(\bold{y})$. Then \textit{zipping} $B$  by a semi-conjugacy $\tau$ which identifies 
 points $r_1\in W^u(\bold{x})$ and $r_2\in W^u(\bold{y})$, that can be joined by 
 a stable leaf included in $B$, one obtains a Smale map isotopic to $f$ whose dynamics 
 is semiconjugated to  $f|_K$ and has $k-1$ special pairs of points. Applying the same process to this new map one 
 can construct a Smale map with $k-2$ special pairs of points. By induction 
 we find a Smale map $f^*$ satisfying the mentioned conclusions.
\end{proof}
\subsection{The dynamics forced by a homoclinic orbit}
In this section we will study a forcing relation on homoclinic orbits to a fixed point.
As we have said in the introduction, in this case one must work with 
the Handel-Miller theory which is a generalization of the Nielsen-Thurston theory
applied to end periodic homeomorphisms \cite{HanMil}.

Let $P=\{p_j:p_j=f^j(p_0),\forall j\in\mathbb{Z}\}$ be a homoclinic orbit to a 
fixed point $p$ of a Smale map $f$ on $D^2$.  As said before let 
$M_P=\operatorname{Int}(D^2)\setminus P$
and let $f_P$ be the restriction of $f$ to $M_P$. 
Next lemma proves that $f_P$ has the topological type of an end-periodic 
automorphism of the noncompact surface with one attracting end and one repelling end.
\begin{lemma}\label{lem:end}
The isotopy class of $f_P$ contains a map which is conjugated to an end-periodic 
homeomorphism  $\overline{f}$ on the non-compact hyperbolic surface 
$S=\mathbb{R}^2\setminus({\mathbb{Z}\times\{0\}})$.
 The map $\overline{f}$ has one attracting end and one repelling end.
\end{lemma}
\begin{proof}
In fact, if splitting open the stable manifold of $p$ by a stable DA-isotopy 
(which was introduced in \cite{Will} and is the inverse process of zipping) 
one obtain a map $f_0$ for which $p$ is an attracting point and such that $f_0$ 
has two saddles points $\bold{x}_1$ and $\bold{x_2}$. See figure \ref{fig:dampas} for 
an example with the horseshoe orbit ${}^{\infty}01010110{}^{\infty}$ homoclinic to 
the point $p=0^{\infty}$.
The stable DA isotopy transforms $P$ into a heteroclinic orbit $Q=\{q_j\}$,
 that is, $\lim_{n\rightarrow +\infty}f_0^n(q_i)=\bold{x}_1$
and $\lim_{n\rightarrow +\infty}f_0^{-n}(q_i)=\bold{x}_2$. See figure \ref{fig:dampas}(b).
\begin{figure}[h]
\centering
\includegraphics[width=110mm,scale=0.6]{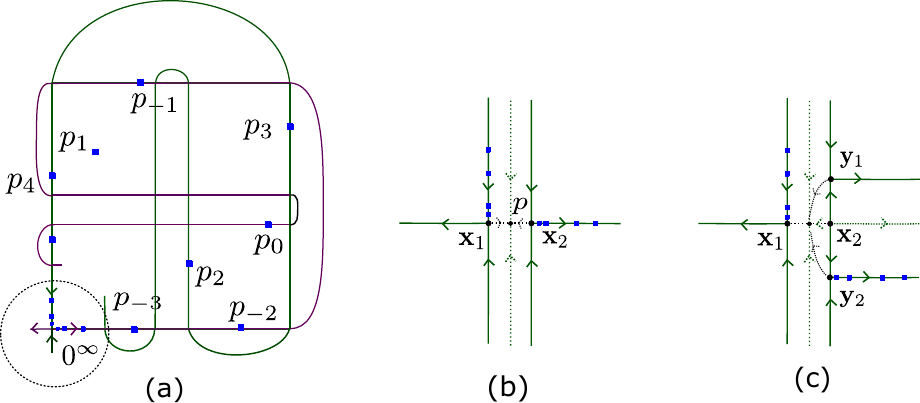}
\caption{DA-isotopies of $f$ and their action around the fixed point $p$.}
\label{fig:dampas}
\end{figure}
Then making an unstable DA isotopy of $f_0$  one get a Smale map $f_1$ for which 
$\bold{x}_2$ is now an repelling point. Furthermore $f_1$ has two saddle fixed points
denote by $\bold{y}_1$ and $\bold{y}_2$. See figure \ref{fig:dampas}(c).
 Thus $Q$ is now transformed in a heteroclinic orbit, that we will also called $Q$, 
satisfying that $\lim_{n\rightarrow +\infty}f_1^n(q_i)=\bold{x}_1$
and $\lim_{n\rightarrow +\infty}f_1^{-n}(q_i)=\bold{y}_2$. By a new isotopy one 
can push $\bold{x}_1$ and $p$ to a point $\bold{t}_1$ in $\partial D^2$, and one can push 
 $\bold{y}_2$ and $\bold{x}_2$ to point $\bold{t}_2\in\partial D^2$, to obtain a 
 map $\overline{f}$ on the disk which has an orbit $Q$ such that 
$\lim_{n\rightarrow +\infty}f_1^n(q_i)=\bold{t}_1$
and $\lim_{n\rightarrow +\infty}f_1^{-n}(q_i)=\bold{t}_2$. Identifying $\operatorname{Int}(D^2)$ with $\mathbb{R}^2$ 
and conjugating so that $Q$ is identified with $\mathbb{Z}\times\{0\}$, 
we obtain a map also called $\overline{f}$ that has a basic compact basic set and
two ends. The surface $S=\mathbb{R}^2\setminus(\mathbb{Z}\times\{0\})$ is hyperbolic  as was shown in \cite[Section 3]{Han99}.
See example \ref{ex:horseshoe}.
\end{proof}
\begin{example} 
In figure \ref{fig:laminaexam}  we have pictured the invariant manifolds of the map 
$\overline{f}$, given by lemma \ref{lem:end}, for the horseshoe homoclinic orbit 
$p_0={}^\infty010\cdot10110^\infty$ which has been represented in figure \ref{fig:dampas}(a).
See section \ref{sec:horseshoedefi} for the description of the Smale horseshoe map.
\begin{figure}[h]
\centering
\includegraphics[width=120mm,scale=0.6]{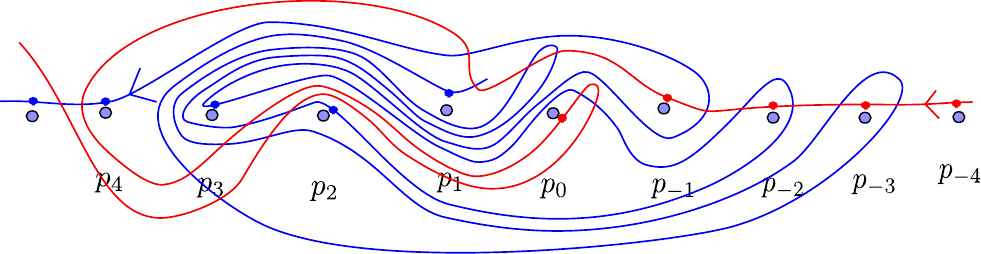}
\caption{Invariant manifolds of the end periodic Smale diffeomorphism $\overline{f}$ constructed
applying lemma \ref{lem:end} to the horseshoe homoclinic orbit ${}^\infty010\cdot10110^\infty$.}
\label{fig:laminaexam}
\end{figure}
\end{example}
We will follow the presentation of the Handel-Miller theory given by Cantwell and Conlon 
in \cite{CanCon2} for end-periodic homeomorphisms, in particular, by lemma \ref{lem:end}, 
we can suppose that  $f_P$  is defined on $S$, has a basic set $K$ and $P=\mathbb{Z}\times\{0\}$. 
Let $\widetilde{S}$ be the universal cover of $S$ which is identified with 
the Poincar\'e disk $\mathbb{D}^2$, $\pi:\widetilde{S}\rightarrow S$ be the recovering map, 
$\widetilde{f}:\widetilde{S}\rightarrow\widetilde{S}$ be a fixed lift of $f$
and $\widehat{f}:S_{\infty}\rightarrow S_{\infty}$ its canonical extension to the
closed unit disk $S_{\infty}:=\mathbb{D}\cup\partial\mathbb{D}$. 
First Cantwell and Conlon have defined the Handel-Miller map $h:S\rightarrow S$ 
isotopic to $f$ such that $h=f$  in the boundary $S^1_{\infty}=\partial \mathbb{D}$. Then they 
extended it using a very interesting construction and were able to find two 
bi-laminations $\Gamma_{\pm}$ and $\Lambda_{\pm}$ on $S$ which are 
invariant by $h$. It follows  that $\Lambda_{\pm}$ is formed by immersed geodesics 
which are complete and whose end-points of their lifts are in $S_{\infty}^1$. 
Denotes the support of $\Lambda_{\pm}$ by $|\Lambda_{\pm}|$.

\begin{definition} \normalfont
The \textit{dynamical core} $\mathcal{C}$ of $f$ is the intersection of the invariant laminations 
$|\Lambda_+|\cap|\Lambda_-|$ of the Handel-Miller map  $h$ associated to $f$.
\end{definition}
The homeomorphism $h$ is not uniquely determined  but  it is unique when it is restricted to the 
core.
\begin{theorem}{\cite[Corollary 10.13]{CanCon2}}\label{thm:core}
The dynamical core $\mathcal{C}$ of $f$ is  uniquely determined up to a topological conjugacy. 
\end{theorem}
The dynamical core relative to a homoclinic orbit $P$, denoted $\mathcal{C}_P$, is 
the dynamical core of $f$ on $S$.  
Since $f$ and $g$ on $S$ are isotopic if and only there are lifts to $\widetilde{S}$
such that $\widehat{f}$ and $\widehat{g}$ agree on the ideal boundary \cite{CanCon} 
it follows that $\mathcal{C}_P$ is a topological invariant under ambient isotopies. 
So $\mathcal{C}_P$ is persistent which means that every $g$ ambiently isotopic to $f$ 
on $S$ has the same core.
Thus one can give the following definition of forcing in this context.
\begin{definition} \normalfont
We will say that the \textit{dynamics forced by $P$} is the 
dynamics of $h:\mathcal{C}_P\rightarrow\mathcal{C}_P$, that is, an $h$-orbit $Q$ is \textit{forced} by $P$, denoted $P\geqslant_2Q$,  if $Q\subset\mathcal{C}_P$.
\end{definition}
Note that this definition involves periodic and non-periodic orbits, indeed, it follows
from \cite[theorem 9.2]{CanCon2} that $h|_{\mathcal{C}_P}$ is  conjugated to a two-ended Markov shift
of finite type.

If a Smale map $f$ does not have  bigons relative to $P$, theorem \ref{thm:first} below proves that 
$\mathcal{C}_P$ agrees with its basic set $K$ up to topological finite-to-one semi-conjugacy. 
It depends of the following result of Grines \cite{GriRep}, extended to a surface 
of infinite type, who has proved analogous conclusions for hyperbolic attractors 
on surfaces of finite type.
\begin{theorem}{\cite{GriRep}}
Let $f$  be a Smale map on $D^2$ with a basic set $K$, and let $P$ be a 
homoclinic orbit to a fixed point $p\in K$.  If $K$ is exteriorly situated 
relative to $P$ without special points then there exist two 
geodesics laminations $(\mathcal{L}^u,\mathcal{L}^s)$, a map $f_0$ on 
$\mathcal{L}^u\cap\mathcal{L}^s$ and a continuous map 
$\iota:\mathcal{L}^u\cap\mathcal{L}^s\rightarrow K$ which is homotopic to the identity such that
\begin{itemize}
\item[(1)]  $\iota(\mathcal{L}^u\cap\mathcal{L}^s)=K$, 
$f\circ \iota|_{\mathcal{L}^u\cap\mathcal{L}^s }=\iota\circ f_0|_{\mathcal{L}^u\cap\mathcal{L}^s}$;
\item[(2)]  the set $B\subset K$ of points $b$ whose preimage $\iota^{-1}(b)$ contains more than 
one point is formed by $K\cap (\cup W^s(p_i^u))$ where $\{p_i^u\}$ is the finite set of the 
$u$-boundary periodic points and $K\cap (\cup W^u(p_i^s))$ where $\{p_i^s\}$ is the finite set of the 
$s$-boundary periodic points;
\item[(3)]  the set $\iota^{-1}(b)$  consists of exactly two points belonging to distint 
boundary geodesics of $\mathcal{L}^u$ or $\mathcal{L}^s$.
\end{itemize}
\end{theorem}
\begin{proof}
It is sufficient to prove the conclusions for the map $\overline{f}:S\rightarrow S$ 
given in lemma \ref{lem:end}.  Hence we can suppose that $f$ is an
end periodic map with a basic saddle set $K$ and two ends, defined 
 on the end periodic surface $S$. 
\begin{figure}[h]
\centering
\includegraphics[width=95mm,scale=0.6]{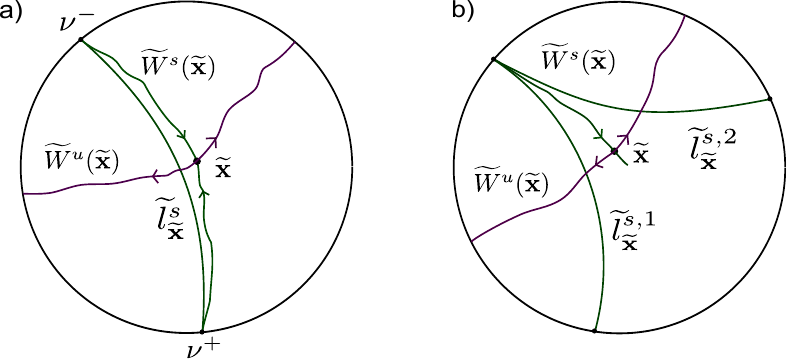}
\caption{The construction of the geodesics $l^{s}_{\bold{x}}$.}
\label{fig:geodesics}
\end{figure}
Using the construction of Grines \cite{GriRep}, for each $\bold{x}\in K$ we have that 
\begin{itemize}
\item if both components of $W^s(\bold{x})\setminus\{\bold{x}\}$ intersect $K$ 
then $W^s(\bold{x})$ lifts in the 
Poincar\'e disk $\widetilde{S}=\mathbb{D}^2$ to a curve $\widetilde{W}^s(\widetilde{\bold{x}})$ 
which has two-endpoints at $S^1_{\infty}$, $\nu^+$ and $\nu^-$. Taking the geodesic tightening 
of $\widetilde{W}^s(\widetilde{\bold{x}})$,  we construct a geodesic 
$\widetilde{l}^s_{\widetilde{\bold{x}}}$, that is, the geodesic passing trough $\nu^+$ and $\nu^-$.
See figure \ref{fig:geodesics}(a). 
\end{itemize}
We set $\widetilde{L}^s$ to the set of 
$\{\widetilde{l}^s_{\widetilde{\bold{x}}}:\bold{x} \textrm{ is a non-boundary periodic point}\}$ 
and $L^s=\pi(\widetilde{L}^s)$. Denote $\mathcal{L}^s$ to the closure of $L^s$ on $S$. 
Note that if $\bold{x}$ is an $u$-boundary periodic point but not 
$s$-boundary then, taking limits of geodesics from both sides of $W^u(\bold{x})$, 
there exists two geodesics $l^{s,1}_{\bold{x}}$ and $l^{s,2}_{\bold{x}}$ 
in $\mathcal{L}^s$, as in figure \ref{fig:geodesics}(b), that can be considered
the geodesic tightening of $W^s(\bold{x})$. 
See \cite[Chapter 9]{GriMedPoc} for a detailed explanation. 

Analogously we state the geodesic lamination $\mathcal{L}^u$. Defining the maps
$\iota$ and $f_0$  by $\iota(l^s_{\bold{x}}\cap l^u_{\bold{x}})=\bold{x}$ and 
$f_0(l^s_{\bold{x}}\cap l^u_{\bold{x}})=l^s_{f(\bold{x})}\cap l^u_{f(\bold{x})}$ 
as in \cite{AraGriEnt}, one can prove that 
$$\iota\circ f_0(l^s_{\bold{x}}\cap l^u_{\bold{x}})=f(\bold{x})=f\circ \iota(l^s_{\bold{x}}\cap l^u_{\bold{x}}).$$
\end{proof} 
Now we will prove the main result of this section.
\begin{theorem}\label{thm:first}
Let $f$  be a Smale map on $D^2$ with a basic set $K$, and let $P$ be a homoclinic 
orbit to a fixed point $p\in K$.  If $K$ is exteriorly situated relative to $P$ 
without special points then $\mathcal{C}_P=K$ up a topological 
finite-to-one semi-conjugacy $\iota:\mathcal{C}_P\rightarrow K$ which is injective on the 
set of non-boundary points.
 \end{theorem}
 \begin{proof}[Proof of Theorem \ref{thm:first}]
We will prove that $\Lambda_+=\mathcal{L}^u$ and $\Lambda_-=\mathcal{L}^s$,
that is, $\mathcal{C}_P=\mathcal{L}^u\cap\mathcal{L}^s$.
It is sufficient to prove that $(W^u(K),W^s(K))$  is a bilamination satisfying the axioms 1--4
given in \cite[Section 10]{CanCon2}. Then by the isotopy theorem \cite[Theorem 10.14]{CanCon2}, 
$(\mathcal{L}^u,\mathcal{L}^s)=(g(W^u(K)),g(W^s(K)))$ is the Handel-Miller lamination, 
where $g(\lambda)=\lambda^g$ is the geodesic tightening of $\lambda$ for any unstable or 
stable leaf as it was defined in the previous paragraph. Thus $K=\mathcal{C}_P$ up to an 
ambient isotopy and a finite-to-one semiconjugacy. 
\begin{itemize}
\item[(i)] We will follow the same argument of the proof of \cite[Theorem 5.3]{LewUre} 
by Lewowicz and Ures. Take a simply closed curve 
$\gamma=\gamma_s\cup\gamma_u$  which is bounded by a stable segment and an
unstable segment. Now let $\bold{x}$ be a non-boundary fixed point of $f^n$ which belongs to $K$ 
and let $\widetilde{\bold{x}}$ be one of its lifts. 
Since $K$ does not have bigons then $W^u(\widetilde{\bold{x}},\widetilde{f})$ intersects 
every lift of $\gamma_s$  at most a point.  
Since $W^u(\bold{x},f)$ intersects infinitely many times $\gamma_s$ then
$W^u(\widetilde{\bold{x}},\widetilde{f})$ has exactly two end-points at the boundary
  $S^1_\infty$ of the 
universal cover of $M$. It proves that $(W^u(K),W^s(K))$ satisfies Axioms 1 and 2.
\item[(ii)] By lemma \ref{lem:end} there exist neighbourhoods $U_+$ and $U_-$ of 
the ends $\bold{t}_1$  and $\bold{t}_2$, such that the sets $J_+=\partial U_+$ and 
$J_-=\partial U_-$ satisfy: (1) $W^u(K)\cup J_-$ and $W^s(K)\cup J_+$ are each sets 
of disjoint pseudo-geodesics, (2) $W^u(K)$ is transverse to $J_+$ and 
$W^s(K)$ is transverse to $J_-$, and (3) no leaf of $W^u(K)$ can meet $J_+$ so as to form
bigon and no leaf of $W^s(K)$ can meet $J_-$ so as to form a bigon.
See example \ref{ex:horseshoe}. Then $(W^u(K),W^s(K))$
satisfies Axiom 4.
\item[(iii)] We will prove that, for any $\gamma_s\subset W^s(K)$, $\gamma_s$ is included 
in the  closure of the 
set of non-escaping component of the positive juncture $J_+$. Suppose that 
$\gamma_s=W^s(\bold{x})$ where $\bold{x}$ is $n$-periodic. Then, by item (ii), $W^u(\bold{x})$ intersects
$J_+$. Thus there exists a segment $\gamma_u\subset W^u(\bold{x})$ with one end-point
 in $\bold{z}\in J_+$ and the other in $\bold{x}$. Since $\lim_{i\rightarrow+\infty}\operatorname{diam}(f^{-in}(\gamma_u))=0$, by the Lambda lemma, given an open set $U$ containing a compact segment 
 of $\gamma_s$, a segment $\alpha\subset J_+$  containing $\bold{z}$ and $\epsilon>0$,
 there exists an $i_0$ such that for all $i\geq i_0$,  $f^{-in}(\alpha)\cap U$
 is $\epsilon$-close to  $\gamma_s\cap U$. Hence $\gamma_s$ is at the closure of the 
 non-escaping components of $J_+$. Since periodic points are dense in $K$, it 
 follows that, for all $\bold{x}\in K$, $\gamma_s=W^s(\bold{x})$ is at the closure of the 
 non-escaping components of $J_+$ as well. 
 Thus $(W^u(K),W^s(K))$ satisfies Axiom 3.
\end{itemize}
\end{proof}

 Thus theorem \ref{thm:first} can be compared with a result by Lewowicz and Ures \cite{LewUre} that 
 proves that the  basic set of a Smale map on a surface of finite topology is included in the 
persistent set given by Handel \cite{Han}, if $K$ is \textit{exteriorly situated},
that is, if there are no bigons.

\begin{example}\label{ex:horseshoe}
Consider the homoclinic orbit  $\Theta$ which corresponds to the 
horseshoe orbit  $p_0={}^\infty010\cdot10^\infty$ which
is represented in figure \ref{fig:tight}(a).
Note that $F$ does not have bigons relative 
to $\Theta$ then, applying theorem \ref{thm:first}, the dynamics forced by this orbit
is the full horseshoe up to a finite-to-one- semiconjugacy.  Note that 
$\overline{f}$  has an infinite orbit of a 1-pronged singularity at the 
elements of $\Theta$. Applying the zero-entropy equivalence relation, stated in
\cite{dCExt}, we obtain the \textit{tight} horseshoe defined on  $\mathbb{S}^2$ which 
is a generalized pseudo-Anosov map $\phi$  having an infinite orbit of  $1$-pronged singularities.
See figure \ref{fig:tight}(b). 
\begin{figure}[h]
\centering
\includegraphics[width=105mm,scale=0.6]{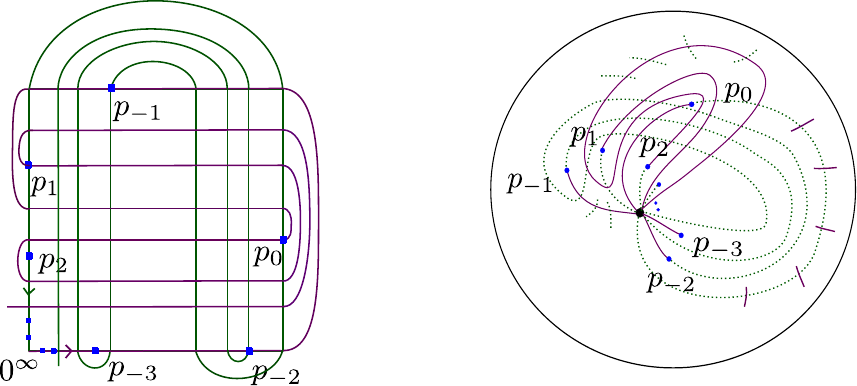}
\caption{The tight horseshoe in the sphere $\mathbb{S}^2$.}
\label{fig:tight}
\end{figure}

The laminations of the end periodic map associated to this homoclinic orbit are
given in figure \ref{fig:lamina}. In this case the map $\overline{f}$ of lemma \ref{lem:end} lies to the isotopy class of $\sigma_{i}^{-2}\tau$
where $\sigma_i$ is a map that interchanges the points $p_i$ and $p_{i+1}$ and $\tau$ is
the translation end map.
\begin{figure}[h]
\centering
\includegraphics[width=120mm,scale=0.6]{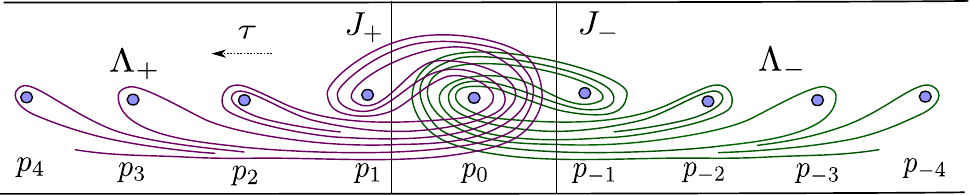}
\caption{The laminations of the end periodic map associated to 
the horseshoe homoclinic orbit ${}^\infty010\cdot10^\infty$.}
\label{fig:lamina}
\end{figure}
\end{example}

\subsection{Pruning theory} 
Pruning is a technique introduced by A. de Carvalho \cite{dC1} for eliminating orbits 
of a homeomorphism in a controlled manner, that is, for destroying dynamics contained 
in the interior of simply connected closed regions that we can define dynamically. 
Here we will use the differentiable version of pruning that the author, in a joint 
work with A. de Carvalho, has developed in the forthcoming paper 
\textit{Differentiable pruning and the hyperbolic pruning front conjecture} \cite{dCarMen}.
The main ideas are the following.

Let $f$ be an orientation-preserving Smale diffeomorphism on a surface $S$ with a basic set $K$.
Let $D$ be a simply connected domain bounded by two segments $\theta_s$  and 
$\theta_u$ with  $\theta_s\subset W^s(p_s)$ and $\theta_u\subset W^u(p_u)$
where $p_s$ and $p_u$ are saddle periodic points in $K$ with periods $n_s$ and $n_u$, respectively. 

\begin{definition}\label{def:pruningconditions} \normalfont
The domain $D$ is a \textit{pruning domain} if its boundary satisfies the
 following properties
 \begin{itemize}
 \item[(i)] $f^n(\theta_s)\cap \operatorname{Int}(D)=\emptyset, \textrm{ for all }n\ge1.$
 \item[(ii)] $f^{-n}(\theta_u)\cap\operatorname{Int}(D)=\emptyset, \textrm{ for all }n\ge1.$
 \end{itemize}
\end{definition}
An easy consequence of definition \ref{def:pruningconditions} is next lemma:
\begin{lemma}\label{lem:boundary}
Let $D$ be a pruning domain. Then the curves $\theta_s$ and $\theta_u$ satisfy:
\begin{itemize}
\item[(i)] Either $p_s\in \theta_s$ or  $F^{n}(\theta_s)\cap \mathring{\theta_s}=\emptyset$, for all $n>0$.
\item[(ii)] Either $p_u\in \theta_u$ or $F^{-n}(\theta_u)\cap \mathring{\theta_u}=\emptyset$, for all $n>0$.
\end{itemize}
\end{lemma}
The differentiable version of the pruning theorem is the following result.
\begin{theorem}[Differentiable Pruning Theorem]\label{theo:pruning}
If $D$ is a pruning domain for $f$ then there exists a diffeomorphism
 $\psi$, isotopic to $f$, satisfying the following properties:
\begin{itemize}
\item[(i)] $\psi$ is a Smale map for which every point of $\operatorname{int}(D)$ is wandering for $\psi$,;
\item[(ii)] the non-wandering set of $\psi$ consists of  
a saddle set  $K_{\psi}$ and a finite set of isolated saddle points,
sinks and sources;   
\item[(iii)] $\psi$, restricted to  $K_{\psi}$, is  semiconjugated  to $f$, 
restricted to a subset $K'\subset K$ by a (at most $4$-to-$1$) semiconjugacy $\zeta:K_{\psi}\rightarrow K'$ satisfying
 $$K'=\{\bold{x}\in K:\operatorname{Orb}(\bold{x},f)\cap\operatorname{Int}(D)=\emptyset\}=K\setminus\bigcup_{i\in\mathbb{Z}} f^i(\operatorname{Int}(D)); \textrm{ and }$$ 
 \item[(iv)] $\zeta$ is injective on the set of non-boundary periodic points.
\end{itemize}
\end{theorem}

\begin{proof}[Sketch of the proof of theorem \ref{theo:pruning}]
The ideia for proving theorem \ref{theo:pruning} is to substitute $f$ by a Smale map
$g_1$ which can be pruned. It is done, if necessary, constructing DA-maps
diffeotopic to  $f$ which are difeomorphisms similar to those ones constructed by 
Williams \cite{Will} for Anosov diffeomorphisms. Thus 
\begin{itemize}
\item we have to splitting open the stable manifold of $p_s$  by a stable DA isotopy to obtain a map 
$g_0$ for which $p_s$ is an attracting periodic point,
\item then to splitting open the unstable manifold of 
$p_u$ by $g_0$, $W^u(g_0,p_u)$, by a unstable DA isotopy to obtain a Smale map $g_1$ isotopic to $f$
for which $p_u$ is a repelling periodic point. 
\end{itemize}
The construction of such attracting and repelling DA isotopies, around $p_s$ and $p_u$ respectively, 
were defined in \cite[Section 7]{BegBonYu} and \cite[Section 2.2.2]{GhrHolSul} in a general setting
using DA-bifurcations.
Thus $g_1$ has a $n_s$-periodic basin of attraction $BA(p_s)$ and a $n_u$-periodic basin of repulsion $BU(p_u)$. 
See figures \ref{fig:deformations}(a) and \ref{fig:deformations}(b).
Moreover $g_1$ has a hyperbolic basic set $K_1$, and there exists a semi-conjugacy $\zeta_0:K_1\rightarrow K$
which is at most 4-to-1 such that 
$$\zeta_0\circ g_1(\bold{x})= f\circ \zeta_0(\bold{x}), \forall \bold{x}\in K_1.$$ 
\begin{figure}[h]
\centering
\includegraphics[width=100mm,scale=0.7]{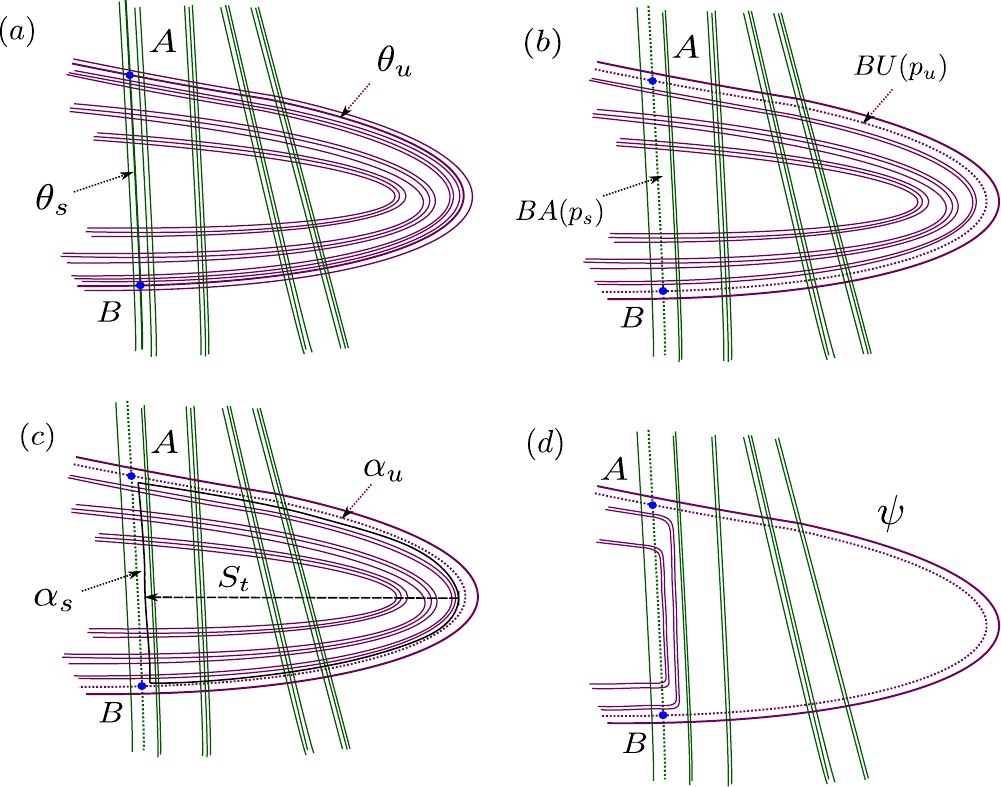}
\caption{A pruning difeotopy.}
\label{fig:deformations}
\end{figure}

Next we define a difeotopy $S_t$ of the identity with support in $\operatorname{Int}(D)$  such that $S_1$ 
takes a curve $\alpha_u\subset BU(p_u)\cap \operatorname{Int}(D)$ to a curve in 
$\alpha_s\subset BA(p_s)\cap\operatorname{Int}(D)$, $S(\alpha_u)=\alpha_s$, in such a way that:
\begin{itemize}
\item $S_1(D\setminus BU(p_u))\subset BA(p_s)$;
\item $S^{-1}(D\setminus BA(p_s))\subset BU(p_u)$.
\end{itemize}
 See figures \ref{fig:deformations}(c) and \ref{fig:deformations}(d). Then  composing $S_t$ with $g_1$ one get a 
diffeotopy $f_t=S_t\circ g_1$ of $f$. Thus $f_t$ takes the non-wandering dynamics of $g_1$ in 
$\operatorname{Int}(D)$ and pushes it to the attracting basin of $p_s$. 
So the \textit{pruning diffeomorphism associated to $D$} will be given by $\psi:=f_1$, the end of 
the diffeotopy. That $\psi$ is a Smale map with a basic set $K_\psi$ included in $K_1$. Thus 
the semi-conjugacy is given by $\zeta:=\zeta_0|_{K_{\psi}}$.
\end{proof}

The following proposition proves that the topology of the invariant 
manifolds of $\psi$ within $D$ is  known: they are just 
deformations by $S_1$ of the invariant manifolds of $f$.
\begin{proposition}\label{prop:topology1}
Let $\gamma_s\subset D$ and $\gamma_u\subset D$ be segments of stable and unstable manifolds
of a point $\bold{x}\in K_\psi$ by $\psi$, respectively. Then
\begin{itemize}
\item[(a)] If $p_s\notin\theta_s$, there exists a segment $\gamma'_u\subset W^u(\bold{x},g_1)$ 
such that $\gamma_u=S_1(\gamma'_u\cap D)$.
\item[(b)] If $p_s\in\theta_s$, there exist a non-negative integer $n_0$ and a segment 
$\gamma'_u\subset W^u(\bold{x},g_1)$ such that $\gamma_u=\psi^{n_0}\circ S_1(\gamma'_u\cap D)$.
 \item[(c)] If $p_u\notin\theta_u$ then there exists a segment $\gamma'_s\subset W^s(\bold{x},g_1)$ 
 such that $\gamma_s=\gamma'_s$.
 \item[(d)] If $p_u\in\theta_u$ then there exist a non-negative integer $n_0\ge1$ and a segment  
 $\gamma'_s\subset W^s(\bold{x},g_1)$ such that $\gamma_s=\psi^{-n_0}(\gamma'_s)$.  
\end{itemize}
\end{proposition}
\begin{proof}
By lemma \ref{lem:boundary}, 
\begin{itemize}
\item[(p1)] either $p_s\notin\theta_s$ and $g_1^n(BA(p_s)\cap \operatorname{Int}(D))\cap \operatorname{Int}(D)=\emptyset,\forall n\geq1$,
or $g_1^{n_s}(BA(p_s)\cap\operatorname{Int}(D))\subset BA(p_s)\cap\operatorname{Int}(D)$; and
\item[(p2)] either $p_u\notin\theta_u$ and $g_1^{-n}(BU(p_u)\cap \operatorname{Int}(D))\cap \operatorname{Int}(D)=\emptyset,\forall n\geq1$
or $g_1^{n_u}(BU(p_u)\cap\operatorname{Int}(D))\subset BU(p_u)\cap\operatorname{Int}(D)$.
\end{itemize}
 Let $\bold{z}\in W^u(\bold{x},\psi)\cap \operatorname{Int}(D)$ and let $n\ge1$ be the smallest integer such that
$\psi^{-n}(\bold{z})\in\operatorname{Int}(D)$ then $\psi^{-i}(\bold{z})=g_1^{-i}\circ S_1^{-1}(\bold{z}), \forall i=0,...,n$.
Since $\operatorname{Int}(D)\setminus BA(p_s)$ is in the basin of an attractor point  for $\psi^{-1}$, it follows
that $g_1^{-n}\circ S_1^{-1}(\bold{z})\in BA(p_s)\cap \operatorname{Int}(D)$. Hence 
$S_1^{-1}(\bold{z})\in g_1^{n}(BA(p_s)\cap\operatorname{Int}(D))$.

By (p1), if $p_s\notin\theta_s$ then 
$\bold{z}\in S_1(g_1^n(BA(p_s)\cap \operatorname{Int}(D)))=g_1^{n}(BA(p_s)\cap\operatorname{Int}(D))$. 
So $\bold{z}\in g_1^{n}(BA(p_s)\cap \operatorname{Int}(D))\cap \operatorname{Int}(D)=\emptyset$. 
It is a contradiction.
 So $\psi^{-i}(\bold{z})=g_1^{-i}\circ S_1^{-1} (\bold{z}), \forall i\ge1.$ This implies that 
$S_1^{-1}(\bold{z})\in W^u(\bold{x},g_1)$, or $\bold{z}\in S_1(W^u(\bold{x},g_1))$. This proves (a).

Now suppose that $p_s\in\theta_s$. Let $n_0$ be the biggest non-negative integer 
such that $\psi^{-n_0}(\bold{z})\in \operatorname{Int}(D)$. So 
$\psi^{-j}(\psi^{-n_0}(\bold{z}))=g_1^{-j}\circ S_1^{-1}(\psi^{-n_0}(\bold{z})), \forall j\ge 1.$
Then $S_1^{-1}(\psi^{-n_0}(\bold{z}))\in W^u(\bold{x},g_1)$. Then $\bold{z}\in \psi^{n_0}\circ S_1(W^u(\bold{x},g_1))$. 
This proves (b).
\begin{figure}[h]
\centering
\includegraphics[width=85mm,scale=0.6]{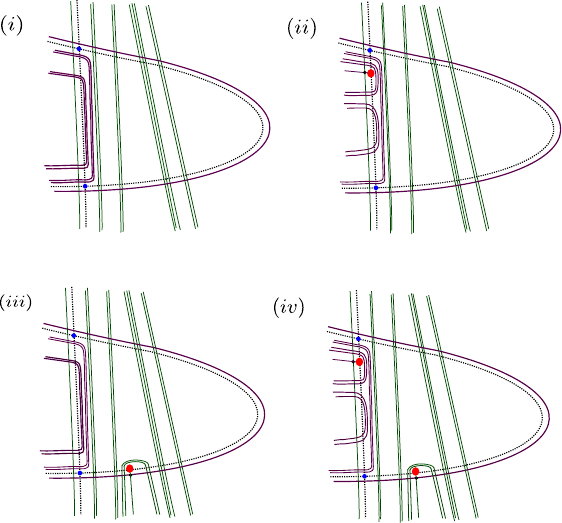}
\caption{Effects of the pruning isotopy on the invariant manifolds of $f$.
The invariant manifolds showed correspond to the pruning diffeomorphism $\psi$.}
\label{fig:invariant}
\end{figure}

Let $\bold{z}\in W^s(\bold{x},\psi)\cap \operatorname{Int}(D)$ and let $n\ge0$ be the smallest integer
 such that $\psi^{n}(\bold{z})\in g_1^{-1}(\operatorname{Int}(D))$. That is $\psi^i(\bold{z})=g_1^i(\bold{z}),\forall i=1,..,n$.
It follows from definition of $S_1$ that $g_1^{-1}(\operatorname{Int}(D)\setminus BU(p_u))$ is included in the
basin  of one attractor point of $\psi$, so $g_1^{n}(\bold{z})\in g_1^{-1}(BU(p_u)\cap \operatorname{Int}(D))$.
If $p_u\notin\theta_u$, by (p2), $g_1^{-n}(BU(p_u)\cap \operatorname{Int}(D))=\emptyset, \forall n\ge1$. 
In this case, we have a contradiction since $\bold{z}\in g_1^{-n-1}(BU(p_u)\cap \operatorname{Int}(D))\cap \operatorname{Int}(D)$. So 
$\psi^i(\bold{z})=g_1^i(\bold{z}),\forall i\ge1.$
This implies that $\bold{z}\in W^s(\bold{x},g_1)$. It proves (c). 

Now suppose that $p_u\in\theta_u$ then let $n_0$ be the biggest 
non-negative integer such that $\psi^{n_0}(\bold{z})\in\operatorname{Int}(D)$. So
$\psi^j\circ \psi^{n_0}(\bold{z})=g_1^j\circ \psi^{n_0}(\bold{z}), \forall j\ge 1.$
Then $ \psi^{n_0}(\bold{z})\in W^s(\bold{x},g_1)$. This implies that 
$\bold{z}\in \psi^{-n_0}(W^u(\bold{x},g_1))$. So we have proved (d).
\end{proof} 

Proposition \ref{prop:topology1} can be useful in applications where it is 
important to know how are modified the stable and unstable manifolds under 
the pruning isotopy. If the derivatives $Df^{n_s}(p_s)$ and $Df^{n_u}(p_u)$ 
have positive eigenvalues, then the invariant manifolds of $\psi$ within $D$ 
are pictured in figure \ref{fig:invariant} for some cases considered in 
proposition \ref{prop:topology1}: (i) if $p_s\notin\theta_s$ and $p_u\notin \theta_u$,
(ii) if $p_u\in\theta_s$ and $p_u\notin\theta_u$, (iii) if $p_s\notin \theta_s$
and $p_u\in\theta_u$, and (iv) if $p_s\in\theta_s$ and $p_u\in\theta_u$. 
The reader is encouraged to draw the invariant manifolds of $\psi$ if one 
of those derivatives has a negative eigenvalue.

From proposition above, for knowing how the bigons are created or destroyed
we just have to concern with the deformation  of the invariant manifolds by $S_1$ 
within $D$. So $\psi$ must just be understood as $f$ without intersections of 
stable and unstable manifolds within $\operatorname{Int}(D)$ (and then 
within all its iterates).

\section{Forcing on homoclinic horseshoe orbits}\label{sec:horseshoe}
In this section we will study certain homoclinic orbits of the Smale 
horseshoe, one of the most famous diffeomorphism in dynamical systems. We 
will determine the dynamical core forced by them exhibiting 
the sequence of pruning maps (or pruning domains) that are sufficient
for eliminating all the bigons relative to these orbits. 

We are only concerned with homoclinic orbits $P=P^w_0$ to $0^\infty$ which have 
as code ${}^\infty01{}^0_1w{}^0_110{}^\infty$ where $w$ is a finite word called the 
\textit{decoration of $P$} (See section below). By a Handel's result, cited in \cite{BoyHall}, 
it is known that the homoclinic orbit ${}^\infty010\cdot10^\infty$ forces every 
horseshoe orbit as it was shown in example \ref{ex:horseshoe}. 
So our study here is devoted to try of making similar 
conclusions for other homoclinic orbits, that is, which are the orbits
 (periodic or not) forced by a given homoclinic one. 
 
 There exist some  methods for determining $\mathcal{C}_P$ for homoclinic orbits. 
 The first one, introduced in Hulme's thesis \cite{Hul}, is a generalization of the 
Bestvina-Handel algorithm to construct pseudo-Anosov-like representatives of certain
homoclinic or heteroclinic braids which can be considered as translation-ends
classes. The second one was stated by Collins in \cite{Coll1} where the trellis 
of a homoclinic or heteroclinic orbit allows to construct a graph representative
which contains the dynamics forced by that orbit and, in certain cases, also 
allows to define a hyperbolic map realising the entropy bound \cite{Coll3}.
There are two main differences between these methods and our technique. 
The first one is the fact that pruning theory  can be applied wherever the 
Collins's method does not work. For example, in \cite[Example 4.3]{Coll3}  
it is pointed out that for the trellis associated to the orbit 
$P={}^\infty01{}^0_1{}^0_110{}^\infty$ there is no 
a minimiser hyperbolic diffeomorphism. As we will see in section \ref{sec:star},
 $P$ corresponds to the \textit{star orbit} $P^{1/3}_0$  for which there exists 
a well-defined pruning region, and so, a Smale map realising $\mathcal{C}_P$.
Other difference consists in that, once we have found a pruning region for 
an individual homoclinic orbit, it is easy to generalize that region for an 
infinite family of  decorations; the reason for it is the observation that 
the pruning regions seem to depend only on the \textit{combinatorics} of 
the orbit and do not depend on the coordinates themselves. It will be interesting to study 
the relation between the combinatorics and the pruning regions in a general 
context.
\subsection{Smale horseshoe}\label{sec:horseshoedefi}
The Smale horseshoe is a well-known hyperbolic diffeomorphism $F$ on $D^2$
which acts as in figure \ref{fig:horseshoe}(a). Its non-wandering dynamics consists 
of an attracting fixed point in the left semi-circle, and a compact basic set $K$
included in the regions $V_0$ and $V_1$. 
Furthermore, there exists a conjugacy between $F$, restricted to $K$, and the
\textit{shift} $\sigma$, defined on the compact set $\Sigma_2=\{0,1\}^\mathbb{Z}$. Thus
each point $\bold{x}\in K$ is represented by a sequence $\bold{s}=(s_i)_{i\in\mathbb{Z}}$ where
$s_i=0$, if $F^i(\bold{x})\in V_0$, and $s_i=1$, if $F^i(\bold{x})\in V_1$.
\begin{figure}[h]
\centering
\includegraphics[width=105mm,height=40mm]{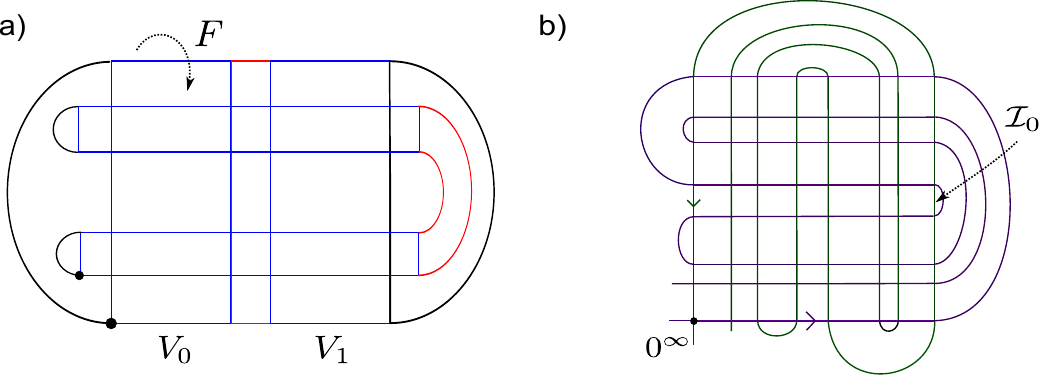}
\caption{The Smale horseshoe map and its homoclinic tangle.}
\label{fig:horseshoe}
\end{figure}

In the following sections, a point $\bold{x}\in K$ will be always represented by its symbol code. 
Each point $\bold{x}\in K$ has two invariant manifolds $W^s(\bold{x})$ and $W^u(\bold{x})$ that are 
dense in $K$ and it will be supposed that they are vertical and horizontal,
respectively. If $\bold{x}=\bold{s}_-\cdot\bold{s}_+$, where $\bold{s}_+$ and $\bold{s}_-$ are unilateral sequences
in the compact space $\Sigma_2^+=\{0,1\}^\mathbb{N}$, 
then we can project $\bold{x}$ along its invariant manifolds to
 the lowest unstable manifold of $0^\infty$ for obtaining ${}^\infty0\cdot\bold{s}_+$,
  and to the leftmost stable manifold of $0^\infty$ for obtaining $\bold{s}_-\cdot0{}^\infty$ .
One can compare two points of $K$ using their positions $\bold{s}_+$ and $\bold{s}_-$ in 
$\Sigma_2^+$ given by the \textit{unimodal order} 
$\geqslant_1$ which relates two symbol sequences  by the following rule: 
If $\bold{s}=s_0s_1\cdots$ and $\bold{t}=t_0t_1\cdots$ are two sequences in $\Sigma_2^+$ 
with $s_i=t_i$ for all $i\leq k$ and $s_{k+1}\neq t_{k+1}$, then $\bold{s}>_1\bold{t}$ when:
\begin{itemize}
\item[(O1)] $\sum_{i=0}^k s_i$ is even and $s_{k+1}>t_{k+1}$, or
\item[(O2)] $\sum_{i=0}^k s_i$ is odd and $s_{k+1}<t_{k+1}$.
\end{itemize}
So $\bold{s}\geqslant_1\bold{t}$ if $\bold{s}=\bold{t}$ or $\bold{s}>_1\bold{t}$. Moreover
let $\bold{s}=\bold{s}_-\cdot\bold{s}_+,\bold{t}=\bold{t}_-\cdot\bold{t}_+ $ be points of
$\Sigma_2$ then
\begin{itemize}
\item if $\bold{s}_+\geqslant_1\bold{t}_+$, we say that $\bold{s}\geqslant_x\bold{t}$, and 
\item if $\bold{s}_-\geqslant_1\bold{t}_-$, we say that $\bold{s}\geqslant_y\bold{t}$.
\end{itemize}
Denote by $\overline{w}$ to the infinite repetition $ww...$ of a word $w$. We will say that a word $w=w_1w_2\cdots w_M$ is 
\textit{even} if $\sum_{i=1}^M w_i$ is even; otherwise $w$ is \textit{odd}.
Let $\widehat{w}=w_Mw_{M-1}\cdots w_1$ be the reversal word of $w$.

Here an horseshoe orbit will be denote by $R$. If $R$ is a periodic orbit then 
the \textit{code} of $R$, denoted by $c_R$ is the symbolic representation 
of the rightmost point of $R$ in the unimodal order. When $R$ is not periodic, 
the code of $R$ can be taken as the symbolic representation of some of its points.

Note that the Smale horseshoe has only one bigon $\mathcal{I}_0$ (and its iterates) formed 
by a stable segment $\theta_s^0$  and an unstable segment $\theta_u^0$ which 
intersect at the homoclinic points  ${}^\infty01{}1\cdot10^\infty$ and 
${}^\infty01{}0\cdot10^\infty$. See figure \ref{fig:horseshoe}(b).

Observe that $F$ does not have special pairs of points.
\subsection{Maximal decorations}
In this section we will study maximal decorations. A decoration $w$
is \textit{maximal} if $w0$ and $w1$ are maximal codes in the 
unimodal order, that is, $\sigma^i(\overline{w0})\leqslant_1\overline{w0}$
and $\sigma^i(\overline{w1})\leqslant_1\overline{w1}$, 
for all $i\geqslant1$. 
\begin{example}\label{ex:max1}
Consider the decoration $w=10$. Since $w0=100$ and $w1=101$ are maximal codes, one 
get that $w$ is maximal. Among the four orbits within the isotopy class 
of  $P_0^{10}$, choose the orbit 
$P={}^\infty010\cdot w110{}^\infty={}^\infty010\cdot 10110{}^\infty$.
\begin{figure}[h]
\centering
\includegraphics[width=115mm,scale=0.6]{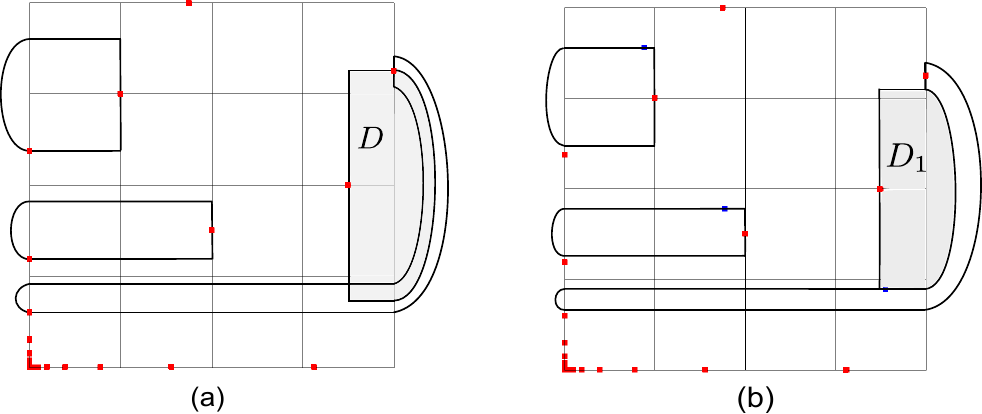}
\caption{The maximal region $D$ is not a pruning domain, but it can be reduced to 
a pruning domain $D_1$.}
\label{fig:exemplomax}
\end{figure}
Note that there exists a maximal region $D$, containing the bigon $\mathcal{I}$,
bounded by a stable segment $\gamma_s$ of ${}^\infty010\cdot 10110{}^\infty$ and an unstable 
segment $\gamma_u$ of ${}^\infty010101\cdot10{}^\infty$.
Since $F^3(\gamma_s)\cap\operatorname{Int}(D)\neq\emptyset$, the domain $D$ is not a
pruning domain. See figure \ref{fig:exemplomax}(a). However
one can decrease $\gamma_s$ to a segment $\theta_s\subset\gamma_s$ such that 
$\theta_s$ and a segment $\theta_u$ bound a pruning domain $D_1$. It is not difficult
to see that $\theta_u$ is included in the unstable manifold of $(w0)^{\infty}=(100)^\infty$.
See figure \ref{fig:exemplomax}(b). Let $\psi$ be the pruning map associated to $D_1$. 
As one will see below, $D_1$ is sufficient for eliminating all the bigons relative to $P^{10}_0$,
that is, for proving that $\psi$ does not have bigons relative to $P^{10}_0$. 
\begin{figure}[h]
\centering
\includegraphics[width=120mm,scale=0.6]{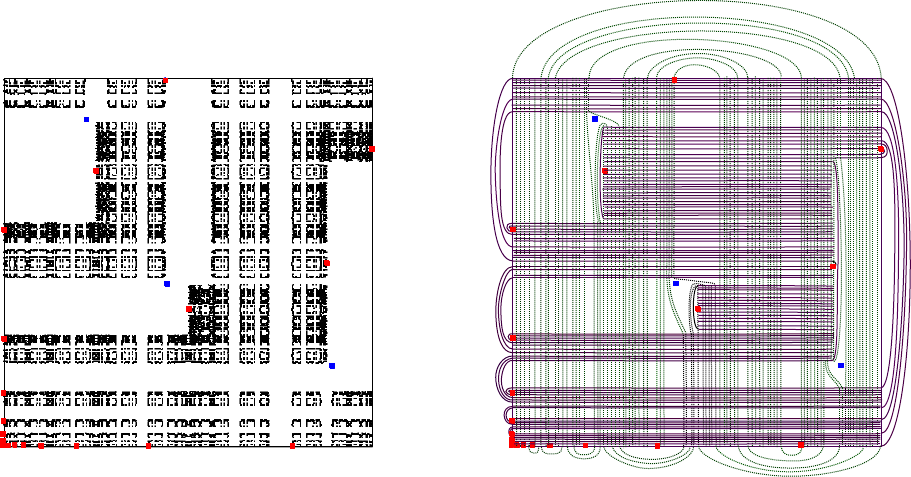}
\caption{A representation of the dynamical core of $P^{10}_0$ in the symbol plane
and the invariant manifolds of $\psi$.}
\label{fig:orbitexem}
\end{figure}
Figure \ref{fig:orbitexem} shows the orbits (of periods less than $16$) that do not 
intersect $D_1$ which are the orbits that belong to the dynamical core forced by $P^{10}_0$
and the invariant manifolds of $\psi$. It was pictured with the program \textit{Domains.exe} 
available in \cite{MenRes} which is a program to prune the Smale horseshoe using at most $5$
pruning domains.
\end{example}

A similar argument to this one used in example above works for an arbitrary maximal decoration. 
\begin{definition}\label{def:prunmax}
Let $w$ be a maximal decoration. We have to distinguish two cases as follows:
\begin{itemize}
\item[(a)] If $w$ is even then define $\mathcal{P}_w=\operatorname{Int}(D_1)$ where 
$D_1$ is the pruning domain bounded by a stable segment 
$\theta_s\subset W^s({}^\infty010\cdot w010{}^\infty)$ and an unstable 
piece $\theta_u\subset W^u((w1)^\infty)$ whose end-points are the heteroclinic points
${}^\infty(w1)\cdot w010^\infty$  and ${}^\infty(w1)w0\cdot w010^\infty$.
\item[(b)] If $w$ is odd then define $\mathcal{P}_w=\operatorname{Int}(D_1)$ 
where $D_1$ is the pruning domain bounded by a stable segment 
$\theta_s\subset W^s({}^\infty010\cdot w110{}^\infty)$ and an unstable piece 
$\theta_u\subset W^u((w0)^\infty)$ whose end-points are the 
heteroclinic points ${}^\infty(w0)\cdot w110^\infty$ 
and ${}^\infty(w0)w1\cdot w110^\infty$.
\end{itemize}
\end{definition}
These sets are represented in the symbol plane in figure \ref{fig:pruningdomains} and are 
called \textit{pruning regions associated to $w$}. 
\begin{figure}[h]
\centering
\includegraphics[width=130mm,scale=0.6]{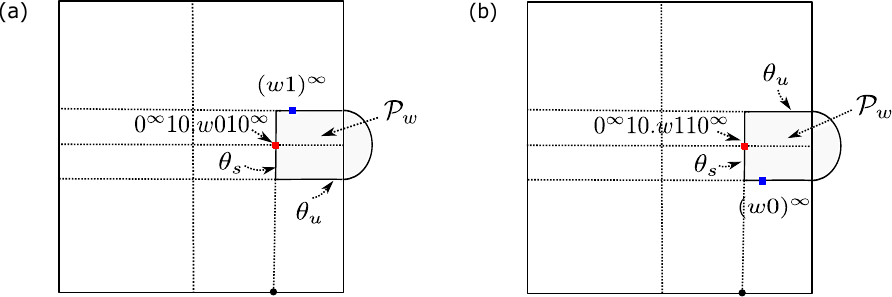}
\caption{The set $\mathcal{P}_w$ when $w$ is even (a) and odd (b).}
\label{fig:pruningdomains}
\end{figure}
The horizontal lines represent the unstable manifolds and the vertical lines
represent the stable manifolds. So the following theorem describes the dynamical core 
forced by $P_0^w$. 
\begin{theorem}\label{thm:main}
Let $w$ be a maximal decoration. 
Then
$$\mathcal{C}_{P^w_0}=\{\bold{x}\in \Sigma_2:\operatorname{Orb}(\bold{x},F)\cap\mathcal{P}_w=\emptyset\}=\Sigma_2\setminus\bigcup_{i\in\mathbb{Z}}\sigma^i(\mathcal{P}_w)$$
up to a finite-to-one semiconjugacy which is injective on the set of non boundary periodic points.
\end{theorem}

The following lemma is needed for proving theorem \ref{thm:main}.
\begin{lemma}\label{lem:maxim}
Let $w=w_1\cdots w_M$ be a maximal decoration. Then
\begin{itemize}
\item[(i)] If $w$ is even then,  for all $i=1,\cdots,M$, 
 $$w_iw_{i+1}\cdots w_M010^\infty\leqslant_1 w010^\infty \textrm{ and }\sigma^i((w1)^{\infty})\leqslant_1 w010{}^{\infty}.$$
\item[(ii)] If $w$ is odd then,  for all $i=1,\cdots,M$,
 $$w_iw_{i+1}\cdots w_M110^\infty\leqslant_1 w110^\infty \textrm{ and }
\sigma^i((w0)^{\infty})\leqslant_1 w110{}^{\infty}.$$
\end{itemize}
\end{lemma}
\begin{proof}
Let us only to prove item (i). If there exists $k<M-i$
such that $w_i\cdots w_{i+k}=w_1\cdots w_{1+k}$ and $w_{i+k+1}\neq w_{k+1}$, then 
both inequalities in (i) hold trivially. Suppose that $w_i\cdots w_M=w_1\cdots w_{M-i+1}$. 
We have two cases:
\begin{itemize}
\item if $w_i\cdots w_M$ is even, since $(w_i\cdots w_M{}^0_1w_1\cdots w_{i-1})^{\infty}\leqslant_1(w0)^{\infty}$, it follows 
that $w_{M-i+2}=1$. Thus $w_i\cdots w_M0\leqslant_1 w$ which implies that $w_i\cdots w_M010^\infty\leqslant_1 w010^\infty$. 
Suppose that $\sigma^{i}((w1)^{\infty})=(a_1\cdots a_{M}a_{M+1})^{\infty}\leqslant_1(w1){}^{\infty}$. 
If $a_1\cdots a_M=w_1\cdots w_M$ then $a_{M+1}=0$, which implies that $\sigma^i((w1)^{\infty})=(w0)^{\infty}$. It is contradiction 
because $(w0)^{\infty}$ is maximal. Thus $a_1\cdots a_M$ and $w_1\cdots w_M$ disagree and then $a_1\cdots a_M\leqslant_1w$. 
Hence $\sigma^{i}((w1)^{\infty})\leqslant_1w010{}^{\infty}$.
\item if $w_i\cdots w_M$ is odd the proof follows the same arguments.
\end{itemize}
\end{proof}
\begin{proof}[Proof of theorem \ref{thm:main}]
Let $P_0^w$ be the homoclinic orbit ${}^\infty010w0\cdot10^\infty$ where 
$w$ is maximal and even with length $M$. 
Consider the domain $D_1$ as in definition \ref{def:prunmax} which clearly contains $\mathcal{I}_0$.
 Note that $D_1$ is bounded by two segments $\theta_s$ and $\theta_u$ 
where $\theta_s$ contains ${}^\infty0 10\cdot w010^\infty$ and $\theta_u$
contains the periodic point $p_u=(w1)^\infty$.
So they are defined by the heteroclinic points $A_1={}^\infty(w1)\cdot w010^\infty$ 
and $A_2={}^\infty(w1)w0\cdot w010^\infty$, that is,
$\theta_s\subset W^s(0^\infty)$ and $\theta_u\subset W^u((w1)^\infty)$.
\begin{figure}[h]
\centering
\includegraphics[width=70mm,scale=0.6]{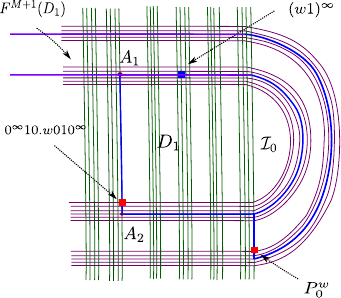}
\caption{The invariant manifolds of $F$.}
\label{fig:pruning}
\end{figure}
\begin{itemize}
\item[(m1)] By lemma \ref{lem:maxim}(i), it follows that $F^i(\theta_s)$ is to 
the left of $w010^\infty$ for all $i=1,\cdots, M-1$. Then $F^i(\theta_s)\cap\operatorname{Int}(D_1)=\emptyset$ for $i=1\cdots M-1$;
\item[(m2)] since $F^M(A_1)={}^\infty(w1) w\cdot010^\infty$, it follows that $F^M(\theta_s)\cap\operatorname{Int}(D_1)=\emptyset$;
\item[(m3)] since $F^{M+1}(A_1)={}^\infty(w1) w0\cdot10^\infty$, it follows that $F^{M+1}(A_1)$ belongs 
to the horizontal unstable leaf that contains $A_2$, and thus $F^{M+1}(\theta_s)\cap\operatorname{Int}(D_1)=\emptyset$. 
\item[(m4)] since $F^i(A_1)={}^\infty(w1) w01\cdots0 \cdot0^\infty, \forall i>M+1$, it follows that $F^i(\theta_s)\cap\operatorname{Int}(D_1)=\emptyset, \forall i >M+1$.
\end{itemize}
Thus for all $i\geqslant 1$, $F^i(\theta_s)\cap\operatorname{Int}(D_1)=\emptyset$. 
Hence $F^n(\mathring{\theta_s})\cap\theta_u=\emptyset$ for
 all $n\ge1$ which implies that $F^{-n}(\theta_u)\cap\mathring{\theta_s}=\emptyset$
  for all $n\ge1$. So if for some $n\ge 1$, 
  $F^{-n}(\theta_u)\cap\operatorname{Int}(D_1)\neq\emptyset$ then 
$F^{-n}(\theta_u)\subset\operatorname{Int}(D_1)$. This is not possible 
since $(w1)^{\infty}\in\theta_u$ and, by lemma \ref{lem:maxim}(i),  
$\operatorname{Orb}((w1)^\infty)\cap \operatorname{Int}(D_1)=\emptyset$.
 Then $D_1$ is a pruning domain. Figure \ref{fig:pruning} shows $D_1$ and 
its $(M+1)$-iterate $F^{M+1}(D_1)$.

By the differentiable pruning theorem \ref{theo:pruning}, there exists a Smale map $\psi$, 
isotopic to $F$,
with a basic set $K_{\psi}$ which is semiconjugated to the set 
\begin{equation}\label{eq:forcedmaximal}
\Sigma_2\setminus\bigcup_{i\in\mathbb{Z}}\sigma^i(\mathcal{P}_w)
=\{\bold{x}\in\Sigma_2:\operatorname{Orb}(\bold{x},F)\cap\mathcal{P}_w=\emptyset\}
\end{equation}
by a semiconjugacy which is injective on the set of periodic orbits. 
As in the proof of theorem \ref{theo:pruning}, $p_u$ is a repelling point of $\psi$ 
and has an repelling basin $BU(p_u)$. See figure \ref{fig:pruningwit}. Since the pruning isotopy is supported in 
$D_1$ and the pruning map uncrosses the invariant manifolds inside $D_1$ 
and its iterates (by proposition \ref{prop:topology1}), it follows that $\psi$ 
has a bigon $\mathcal{I}'$ which is a deformation of $\mathcal{I}_0$ and does not have any bigon 
inside $D_1$. The  bigon $\mathcal{I}'$ has the point ${}^\infty010\cdot w010^\infty$ in its boundary. 
See Figure \ref{fig:pruningwit}. 

A combinatorial analysis on the iterates of the pruning domain
 gives us the code of the boundary periodic points of $\psi$.
\begin{proposition}\label{prop:pronged}
The basic set $K_{\psi}$ of $\psi$ has an $u$-boundary $3(M+1)$-periodic orbit $B_w$ 
whose unstable manifolds define a principal region of $3$ sides that contains a
 $(M+1)$-periodic singularity $S_w$. 
If $w$ is even then the code of $B_w$ is $w1w0w1$ and the code of $S_w$ is $w1$, 
and if $w$ is odd then the code of $B_w$ is $w0w1w0$ and the code of $S_w$ is $w0$.
\end{proposition}
\begin{proof}
Suppose that $w$ is even. Iterating $(M+1)$ times the domain $D_1$, it follows 
that the invariant manifolds are 
uncrossed in the region $C^1\supset D_1$ defined by
$$C^1=\{p:w010{}^\infty\leqslant_x p\leqslant_x10{}^{\infty}, A_2\leqslant_yp\leqslant_y{}^{\infty}(w1)w0w1\cdot w010{}^{\infty}\}$$
which is between the unstable leaf of $A_2$ and the unstable leaf $l_1$ passing through 
the point ${}^{\infty}(w1)w0w1\cdot w010{}^{\infty}$. See figure \ref{fig:exampleboundary}(a). 
\begin{figure}[h]
\centering
\includegraphics[width=140mm,scale=0.6]{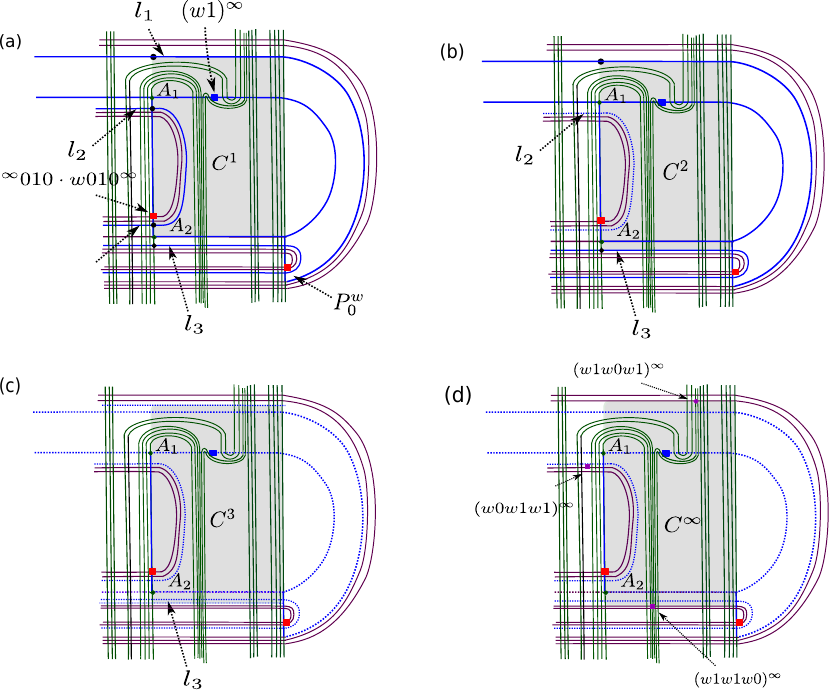}
\caption{}
\label{fig:exampleboundary}
\end{figure}
Iterating $M+1$ times the leaf 
$l_1$ it follows that the uncrossing happen between $l_1$ and the unstable leaf $l_2$
passing through the points ${}^{\infty}(w1)w0w1w0\cdot w010{}^{\infty}$ and 
${}^{\infty}(w1)w0w1w1\cdot w010{}^{\infty}$. Iterating $M+1$ times the leaf $l_2$ 
one can conclude that the invariant manifolds are uncrossed in the region $C^2\supset C^1\supset D_1$ defined by
$$C^2=\{p:w010{}^\infty\leqslant_x p\leqslant_x10{}^{\infty}, {}^{\infty}(w1)w0w1w1w0\cdot w010{}^{\infty}\leqslant_yp\leqslant_y{}^{\infty}(w1)w0w1\cdot w010{}^{\infty}\}$$
which is between the leaves $l_1$
and $l_3$ that contains the point ${}^{\infty}(w1)w0w1w1w0\cdot w010{}^{\infty}$. See figure \ref{fig:exampleboundary}(b). 

Repeating the process a second time using the leaves $l_3$ and $l_1$ one can show that 
the invariant manifolds are uncrossed in the region 
$$C^3=\{p:\begin{array}{c}
w010{}^\infty\leqslant_x p\leqslant_x10{}^{\infty}, \\
{}^{\infty}(w1)w0(w1w1w0)^2\cdot w010{}^{\infty}\leqslant_yp\leqslant_y{}^{\infty}(w1)w0(w1w1w0)w1\cdot w010{}^{\infty}\\
\end{array}\}$$
between unstable leaves passing through the points 
${}^{\infty}(w1)w0(w1w1w0)^2\cdot w010{}^{\infty}$ and ${}^{\infty}(w1)w0(w1w1w0)w1\cdot w010{}^{\infty}$. 
See figure \ref{fig:exampleboundary}(c). Repeating it $n+1$ times one can prove that the invariant manifolds are 
uncrossed in the region 
$$C^{n+1}=\{p:\begin{array}{c}
w010{}^\infty\leqslant_x p\leqslant_x10{}^{\infty}, \\
{}^{\infty}(w1)w0(w1w1w0)^n\cdot w010{}^{\infty}\leqslant_yp\leqslant_y{}^{\infty}(w1)w0(w1w1w0)^{n-1}w1\cdot w010{}^{\infty}\\
\end{array}\}.$$

In the limit the invariant manifolds are uncrossed in the region 
$$C^{\infty}=\{p:\begin{array}{c}
w010{}^\infty\leqslant_x p\leqslant_x10{}^{\infty}, \\
{}^{\infty}(w1w1w0)\cdot w010{}^{\infty}\leqslant_yp\leqslant_y{}^{\infty}(w1w1w0)w1\cdot w010{}^{\infty}\\
\end{array}\}.$$
between the unstable leaves passing through
${}^{\infty}(w1w1w0)\cdot w010{}^{\infty}$ and ${}^{\infty}(w1w1w0)w1\cdot w010{}^{\infty}$
which contain the periodic points $(w1w1w0)^{\infty}$ and $(w1w0w1)^{\infty}$, respectively.
Thus the orbit of $(w1w1w0)^{\infty}$ is an unstable boundary periodic orbit. See figure \ref{fig:exampleboundary}(d). 
\end{proof}
\begin{figure}[h]
\centering
\includegraphics[width=150mm,scale=0.6]{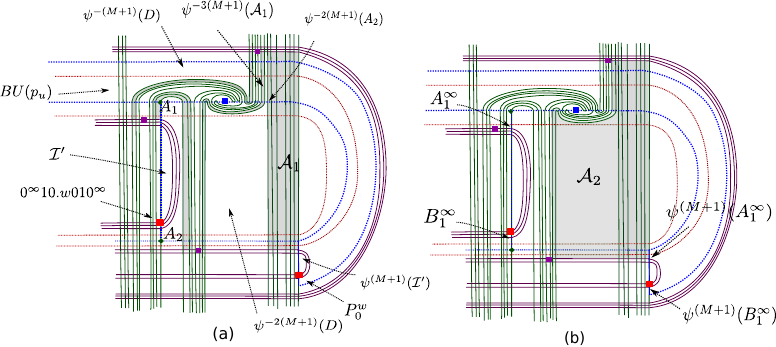}
\caption{The invariant manifolds of $\psi$.}
\label{fig:pruningwit}
\end{figure}
Let us to finish the proof of theorem \ref{thm:main} proving that $\psi$ has no bigons.
From proposition \ref{prop:pronged}, the region 
\begin{equation*}
\mathcal{A}_1=[{}^{\infty}(w1w0w1)\cdot w1w0w010{}^{\infty},{}^{\infty}(w1w0w1)\cdot10{}^{\infty}, {}^{\infty}(w1w1w0)\cdot10{}^{\infty}, {}^{\infty}(w1w1w0)\cdot w1w0w010{}^{\infty}],
\end{equation*} 
which is the rectangle whose boundary is formed by segment of invariant manifolds, does not contains bigons.
Using proposition \ref{prop:topology1} and the combinatorics of the edges of $\mathcal{A}_1$, one can see that 
 the rectangle
\begin{align*}
\psi^{-3(M+1)}(\mathcal{A}_1)&=[{}^{\infty}(w1w0w1)\cdot w1w0w1w1w0w010{}^{\infty},{}^{\infty}(w1w0w1)\cdot w1w0w110{}^{\infty},\\
& {}^{\infty}(w1w1w0)\cdot w1w1w010{}^{\infty}, {}^{\infty}(w1w1w0)\cdot w1w1w0 w1w0w010{}^{\infty}],
\end{align*}
has a stable side included in the stable side of $\psi^{-2(M+1)}(D)$. See figure \ref{fig:pruningwit}(a).
Since $\mathcal{A}_1$ has also a stable side in $\psi^{-2(M+1)}(D_1)$, one can extend $\mathcal{A}_1$ to a bigger 
rectangle 
\begin{align*}
\mathcal{A}_2&=[{}^{\infty}(w1w0w1)\cdot w1w0w1w1w0w010{}^{\infty},{}^{\infty}(w1w0w1)\cdot10{}^{\infty}, \\
&{}^{\infty}(w1w1w0)\cdot10{}^{\infty}, {}^{\infty}(w1w1w0)\cdot w1w1w0 w1w0w010{}^{\infty}]
\end{align*} 
without bigons. See figure \ref{fig:pruningwit}(b).  Repeating this process  for $\mathcal{A}_2$, we can construct
a rectangular region $\mathcal{A}_3\supset \mathcal{A}_2$ such that 
\begin{align*}
\mathcal{A}_3&=[{}^{\infty}(w1w0w1)\cdot (w1w0w1)^2w1w0w010{}^{\infty},{}^{\infty}(w1w0w1)\cdot10{}^{\infty}, \\
&{}^{\infty}(w1w1w0)\cdot10{}^{\infty}, {}^{\infty}(w1w1w0)\cdot (w1w1w0)^2 w1w0w010{}^{\infty}].
\end{align*} 
Repeating this process ad infinitum one obtains a region 
\begin{align*}
\mathcal{A}_{\infty}&=[{}^{\infty}(w1w0w1)\cdot (w1w0w1){}^{\infty},{}^{\infty}(w1w0w1)\cdot10{}^{\infty}, \\
&{}^{\infty}(w1w1w0)\cdot10{}^{\infty}, {}^{\infty}(w1w1w0)\cdot (w1w1w0){}^{\infty}]
\end{align*} 
without bigons. It is clear that $\mathcal{A}_{\infty}$ is periodic of period $3(M+1)$ and then 
there are no bigons in the region $\mathcal{A}_{\infty}\cup \psi^{-(M+1)}(\mathcal{A}_{\infty})\cup\psi^{-2(M+1)}(\mathcal{A}_{\infty})$.
Thus one can conclude that $\psi$ does not have bigons relative to $P^w_0$. 
  By theorem \ref{thm:first} and  (\ref{eq:forcedmaximal}), the orbits which do not intersect $\operatorname{Int}(D_1)$ 
 belong to the dynamics forced by $P^w_0$ up a finite-to-one semiconjugacy injective on  the set of 
 non-boundary periodic points.

By proposition \ref{prop:topology1}, the unstable invariant manifolds of $\psi_1$ are just deformations of  
unstable invariant manifolds of $F$ and so $\mathcal{I}_0$ is pushed by the pruning isotopy to  $\mathcal{I}'$. 
Using proposition \ref{prop:pronged}, the crossings between invariant manifolds situated in $[A_1,B_1]_s$ are 
situated in the segment $[A^{\infty}_1,B^{\infty}_1]_s$ where $A^{\infty}_1$ and $B^{\infty}_1$ are 
points at $W^u((w0w1w1)^{\infty})$. See figure \ref{fig:pruningwit}(b). Iterating $(M+1)$ times, one see 
that  there are no bigons in $\psi^{M+1}([A^{\infty}_1,B^{\infty}_1]_s)$ and so there $\psi$ does not have bigons.

If $w$ is odd the proof follows the same lines with minor modifications. 
\end{proof}

Figure \ref{fig:exampleunstable} shows how proposition \ref{prop:pronged} works when 
$w$ is even: three points with codes $w1w0w1$, $w0w1w1$ e $w1w1w0$ define a ideal 
polygon of $3$ sides that contains the point $(w1)^{\infty}$, creating a $3$-pronged singularity. 
\begin{figure}[h]
\centering
\includegraphics[width=75mm,scale=0.6]{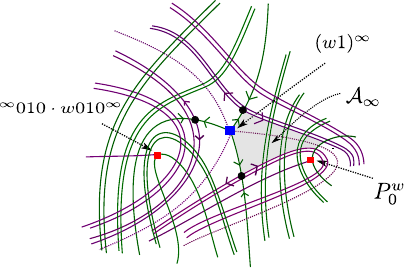}
\caption{If  $w$ is even, the orbit of the unstable boundary point with code $w1w0w1$ 
can be collapsed to the periodic orbit with code $w1$.}
\label{fig:exampleunstable}
\end{figure}

\begin{example}
Consider the homoclinic orbit $P^{10}_0$ of the point $p_0={}^{\infty}010\cdot10110{}^{\infty}$ 
with maximal decoration of example \ref{ex:max1}, and let $\psi$ be its hyperbolic pruning map 
associated which, by proposition \ref{prop:pronged}, has a period $3$ $3$-pronged singularity at the periodic orbit $100$. 
\end{example}

From theorem \ref{thm:main} it is easy to determine the forcing relation between 
two homoclinic orbits $P_0^w$ and $P^{w'}_0$: 
if the orbit of $P^{w'}_0$ does not intersect the pruning region  $\mathcal{P}_w$ 
then $P_0^w$ forces $P^{w'}_0$.
For example, a consequence of theorem \ref{thm:main} is the following 
corollary.
\begin{corollary}\label{cor:maximal}
Let $w$ and $w'$ be two maximal decorations satisfying $w\geqslant_1 w'$ and 
$\widehat{w}\geqslant_1\widehat{w'}$.
 Then $P^{w}_0\geqslant_2 P^{w'}_0$. 
\end{corollary}
\begin{proof}
From conditions $w\geqslant_1 w'$ and $\widehat{w}\geqslant_1\widehat{w'}$, it holds 
that $\mathcal{P}_{w}\subset\mathcal{P}_{w'}$  and then
 $\operatorname{Orb}(P^{w'}_0)\cap\mathcal{P}_w=\emptyset$. Thus, by theorem \ref{thm:main},
$P^{w}_0\geqslant_2 P^{w'}_0$.
\end{proof}
 It is interesting to compare corollary \ref{cor:maximal} with one-dimensional dynamics
where, by the Milnor-Thurston kneading  theory \cite{MilThu}, the condition  $w\geqslant_1w'$ implies
that $w010{}^\infty$ forces the existence of $w'010{}^\infty$.
This result is also closely related to  \cite[Theorem 4.8]{HolWhi} by Holmes and Whitley 
where it was  proved that, for small values of $b>0$ in the H\'enon family 
$H_{a,b}(x,y)=(a-x^2-by,x)$, if $w$ and $w'$ are maximal codes and 
$\widehat{w}01\geqslant_1\widehat{w'}01$ then $P^w_0$ appears after $P^{w'}_0$.

\begin{example}
By corollary \ref{cor:maximal}, it follows that 
$${}^\infty01{}^0_110^m1a10^n1{}^0_110^\infty\geqslant_2 {}^\infty01{}^0_110^l1b10^k1{}^0_110^\infty$$
for $m>l>k$, $m>n>k$ and any words $a$ and $b$ such that $10^m1a10^n1{}^0_1$ and $10^l1b10^k1{}^0_1$ are maximal codes.
\end{example}

\subsection{Concatenations of NBT decorations}\label{sec:nbt}
Another group of decorations for which we know their pruning regions are these ones
called \textit{concatenations of NBT codes}.
To define this kind of orbits we 
need the notion of NBT orbits introduced in \cite{Hall} by T. Hall.
\begin{definition} \normalfont
To every rational number $q$ in $\widehat{\mathbb{Q}}:=\mathbb{Q}\cap(0,\frac{1}{2})$ 
we associate a symbol code in the following way: If $q=\frac{m}{n}$, let $L_q$ be the
 straight line segment joining the origin $(0,0)$ and the point $(n,m)$ in $\mathbb{R}^2$. 
 Then construct a finite word $c_q =s_0s_1...s_n$ by the following rule:
\begin{equation}
s_i=\left\{ \begin{array}{cc}
1 & \textrm{ if $L_q$ intersects some line $y=k, k\in\mathbb{Z}$, for $x\in(i-1,i+1)$} \\
0 & \textrm{ otherwise }\\
\end{array}
\right.
\end{equation}
\end{definition}
It implies that $c_q$ is palindromic and has the form 
$c_q=10^{\mu_1}1^20^{\mu_2}1^2\cdots1^20^{\mu_{m-1}}1^20^{\mu_m}1$. 
The $(n+2)$-periodic orbit $P_q$, having $c_q0$ or $c_q1$ as code, is called a \textit{NBT orbit}.

The following is the main result of \cite{Hall} which claims that the forcing
order restricted to NBT orbits coincides with the unimodal order.
\begin{theorem}[Hall]\label{thm:hallnbt}
Let $q,q'\in\widehat{\mathbb{Q}}$. Then
\begin{itemize}
\item[(i)]  $P_q$ is quasi-one-dimensional, that is, $P_q\geqslant_1R \Longrightarrow P_q\geqslant_2 R$.
\item[(ii)] $q\leq q' \Longleftrightarrow P_q\geqslant_1 P_{q'} \Longleftrightarrow P_q\geqslant_2 P_{q'} $.
\end{itemize}
\end{theorem}
A decoration $w$ is a \textit{concatenation of NBT codes} if there exists a 
finite sequence $\{q_i\}_{i=1}^n$ of distinct rational numbers in $\widehat{\mathbb{Q}}$
such that 
\begin{equation}
w=c_{q_1}0c_{q_2}0\cdots0c_{q_n}.
\end{equation}

We will give conditions for constructing a pruning diffeomorphism $\psi$ relative to 
the homoclinic orbit $P^w_0$, whose code is ${}^\infty01\cdot0c_{q_1}0c_{q_2}0\cdots c_{q_n}010^\infty$,
with a well-defined  pruning region. At first we have to organize the rational numbers in the 
real line $q_{i_1}>q_{i_2}>\cdots>q_{i_n}$. By theorem \ref{thm:hallnbt}, their codes
satisfy $c_{q_{i_1}}0\leqslant_1 c_{q_{i_2}}0\leqslant_1\cdots\leqslant_1 c_{q_{i_n}}0$. 
Denote $c_{q_0}=c_{q_{n+1}}=10^\infty$.

\begin{definition}[Limiting points]\label{defi:plist} \normalfont
Let $C_i$ be the point of the orbit $P^w_0$ whose code is ${}^{\infty}010c_{q_1}0\cdots0 c_{q_{i-1}}0\cdot c_{q_i}0\cdots 0c_{q_n}010{}^{\infty}$.
A point $C_i$, with $i=1,\cdots,n$, is a \textit{limiting point} if there 
are no points $C_k$ in the region 
$$\mathcal{R}_{ii}=\{\bold{x}=\bold{s}_-\cdot\bold{s}_+\in\Sigma_2:\begin{array}{c}
c_{q_i}0\cdots c_{q_n}010^{\infty}\leqslant_1 \bold{s}_+\leqslant_1 10{}^{\infty}, \\
0c_{q_{i-1}}0c_{q_{i-2}}\cdots 0c_{q_1}010^{\infty}\leqslant_1\bold{s}_-\leqslant_1010{}^{\infty}\\
\end{array}\}.$$ 
The \textit{successor of a limiting point} $C_i$ is  another limiting point 
$C_j$ such that $q_j<q_i$ and 
the region 
$$\mathcal{R}_{ij}=\{\bold{x}=\bold{s}_-\cdot\bold{s}_+\in\Sigma_2:\begin{array}{c}
c_{q_i}0\cdots c_{q_n}010^{\infty}\leqslant_1 \bold{s}_+\leqslant_1 10{}^{\infty}, \\
0c_{q_{j-1}}0c_{q_{j-2}}\cdots 0c_{q_1}010^{\infty}\leqslant_1\bold{s}_-\leqslant_1010{}^{\infty}\\
\end{array}\}$$
 does not contain any other point $C_k$. 
\end{definition}
\begin{figure}[h]
\centering
\includegraphics[width=100mm,scale=0.6]{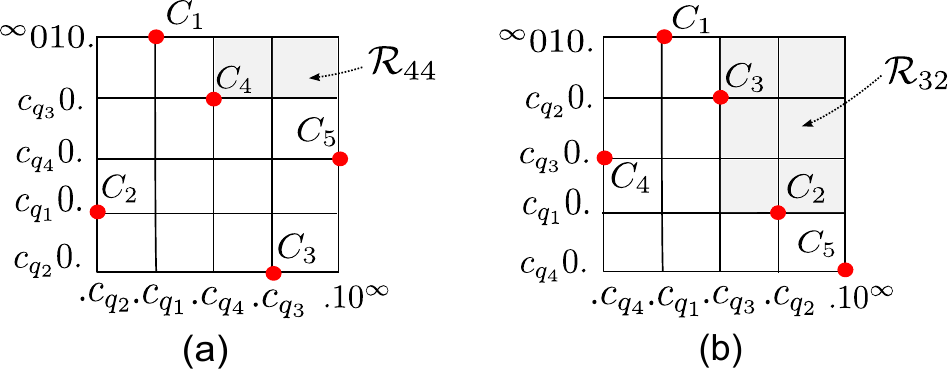}
\caption{The list in $(a)$ is a P-list, but this one in (b) is not.}
\label{fig:prunable}
\end{figure}
Denote by $\mathbb{L}$ the set of 
all the limiting points. It follows that $C_1$ is always a limiting point and we will 
consider that $C_{n+1}={}^{\infty}010c_{q_1}0\cdots0c_{q_n}0\cdot 10^\infty$ is a 
limiting point as well. The limiting points in figure \ref{fig:prunable}(a) 
are $C_1$, $C_4$  and $C_5$ and the filled region is $\mathcal{R}_{44}$. 
The successors of $C_1$ and $C_4$ are 
$C_4$ and $C_5$, respectively. In figure \ref{fig:prunable}(b) the limiting points
are $C_1,C_2$, $C_3$ and $C_5$ and successors of $C_1, C_2$ and $C_3$ are $C_3$, $C_5$ and $C_2$,
respectively.  The filled region is $\mathcal{R}_{32}$.

\begin{definition}[P-list] \normalfont
We say that $\{q_i\}_{i=1}^n$ is a \textit{P-list} if for every limiting point 
$C_i$, with $1\leq i\leq n$, and its successor $C_j$ it holds that $i<j\leq n+1$ and the points $C_{k}$, for all $i<k<j$, 
are not limiting points. When the successor of $C_i$ is $C_{n+1}$, we require 
that $C_k$, for all $i<k\le n$, do not be limiting points. 
\end{definition}
All lists with $2$ or $3$ elements are P-lists. The list in figure \ref{fig:prunable}(a) 
is a P-list. For the list in figure \ref{fig:prunable}(b), since the successor of 
$C_1$ is $C_3$, but $C_2$ is also a limiting point, it follows that it 
is not a P-list.

When $\{q_i\}_{i=1}^n$ is a P-list, there exists a pruning region defined
as follows.
\begin{definition}[Pruning domains] \normalfont
Let $\{q_i\}_{i=1}^n$ be a P-list. For every limiting point $C_i$ and its successor $C_j$
define the domain $D_i$ which is bounded 
by a stable segment $\theta_s^i$ containing the homoclinic point 
 $C_i$
and an unstable segment $\theta_u^i$ passing through the periodic point 
$T_i=(c_{q_i}0\cdots0c_{q_{j-1}}1)^\infty$. The end-points of $\theta_s^i$ and $\theta_u^i$ are 
the heteroclinic points ${}^\infty(c_{q_i}0c_{q_{i+1}}0\cdots c_{q_{j-1}}1)\cdot c_{q_i}0\cdots0c_{q_n}010^\infty$ and $${}^\infty(c_{q_i}0c_{q_{i+1}}0\cdots c_{q_{j-1}}1)c_{q_i}0c_{q_{i+1}}0\cdots c_{q_{j-1}}0\cdot c_{q_i}0\cdots0c_{q_n}010^\infty.$$
\end{definition}
 \begin{definition}[Pruning region]\label{def:front}\normalfont
 The \textit{pruning region of the list $\{q_i\}_{i=1}^n$ or of the decoration 
$w=c_{q_1}0\cdots0c_{q_n}$} is the set
\begin{equation}
\mathcal{P}_w=\bigcup_{C_i \in\mathbb{L}} \operatorname{Int}(D_i).
\end{equation}
\end{definition}

Now one can prove the following result.
\begin{theorem}\label{thm:nbthomo}
Let $\{q_i\}_{i=1}^n$ be a P-list and let $w=c_{q_1}0c_{q_2}\cdots0c_{q_n}$ be its
associated concatenation.  Then 
$\mathcal{C}_{P^w_0}=\{\bold{x}\in \Sigma_2:\operatorname{Orb}(\bold{x},F)\cap\mathcal{P}_w=\emptyset\}$
up to a finite-to-one semi-conjugacy which is injective on the set of non boundary periodic points.
\end{theorem}
\begin{proof}
Let $C_j$ be the successor of $C_1$ and consider the domain  $D_1$ as in definition \ref{def:front} which is bounded by a stable segment 
$\theta^1_s\subset W^s(C_1)$ and an unstable segment $\theta^1_u\subset W^u(T_1)$ 
where $T_1=(c_{q_1}0c_{q_2}0\cdots 0c_{q_{j-1}}1)^\infty$. Let $M_1$ be the period of $T_1$. It is useful to note 
that $D_1$ and $F^{M_1}(D_1)$ have the same behaviour as the domains $D_1$ and $F^{M+1}(D_1)$ for maximal decorations. 
See figure \ref{fig:concatenanbt}(a) for a geometrical explanation of these facts.

To see that $D_1$ is a pruning domain note that the end-points of $\theta^1_s$ are 
the heteroclinic points $A_1={}^\infty(c_{q_1}0c_{q_2}0\cdots c_{q_{j-1}}1)\cdot c_{q_1}0\cdots0c_{q_n}010^\infty$ 
and $$A_2={}^\infty(c_{q_1}0c_{q_2}0\cdots c_{q_{j-1}}1)c_{q_1}0c_{q_2}0\cdots c_{q_{j-1}}0\cdot  c_{q_1}0\cdots0c_{q_n}010^\infty.$$ 
The iterates of $A_1$ and $A_2$ can be of the forms:
$\cdots 0^l11\cdot0^{l'},\cdots 0^l\cdot110^{l'}\cdots, \cdots0^l1\cdot10^{l'}\cdots, \cdots c_{q_k}\cdot0c_{q_{k+1}}\cdots$  
and $\cdots c_{q_k}0\cdot c_{q_{k+1}}\cdots \textrm{ with }1<k<j. $
The four forms to the left are clearly disjoint from $D_1$. Since $C_k$ is 
not a limiting point then the fifth form is disjoint from $\operatorname{Int}(D_1)$.
Thus $F^n(\theta^1_s)\cap \operatorname{Int}(D_1)=\emptyset$ for all $n\geq1$. As in the 
proof of theorem \ref{thm:main}, it implies that $D_1$ is a pruning domain. 

\begin{figure}[h]
\centering
\includegraphics[width=125mm,scale=0.6]{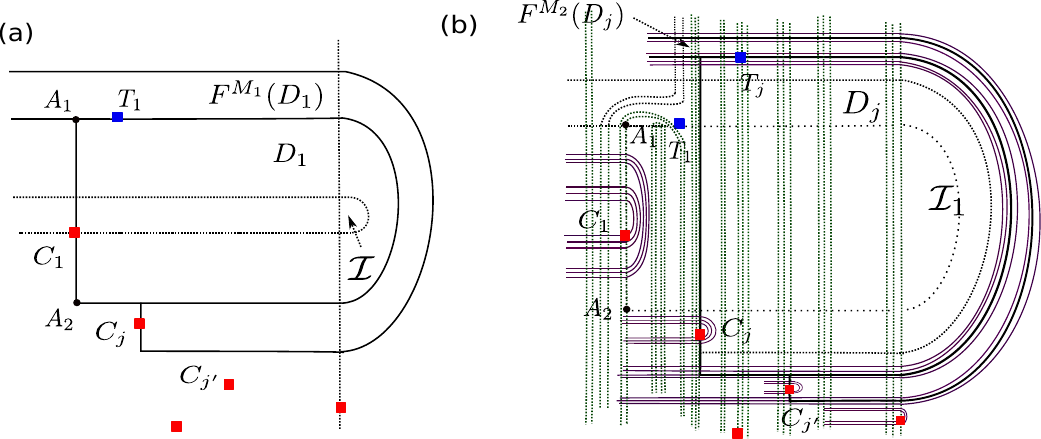}
\caption{(a) The first pruning domain $D_1$ and its iterate $f^{M_1}(D_1)$. 
(b) The second pruning domain defined by the successor $C_j$.}
\label{fig:concatenanbt}
\end{figure}

Let $\psi_1$ be the pruning Smale map associated constructed using $D_1$. Thus $\psi_1$ has a hyperbolic basic set
$K_{\psi_1}$ such that $K_{\psi_1}=\{\bold{x}\in\Sigma_2: F^i(\bold{x})\notin \operatorname{Int}(D_1)\}$
up to a finite-to-one semiconjugacy. 
As in the case of maximal decorations (proposition \ref{prop:pronged}), it follows that
$K_{\psi_1}$ has an unstable periodic point with code $B_{1,j}1B_{1,j}0B_{1,j}1$ where
$B_{1,j}=c_{q_1}0\cdots c_{q_{j-1}}$.
Again as for maximal decorations it can be proved that $\psi_1$ has a bigon $\mathcal{I}_1$ situated to the 
right of $C_j$, and formed by a stable segment of ${}^{\infty}0\cdot 10{}^{\infty}$ and an unstable segment
of $(B_{1,j}1B_{1,j}0B_{1,j}1)^{\infty}$. See figure \ref{fig:concatenanbt}(b).
\begin{figure}[h]
\centering
\includegraphics[width=125mm,scale=0.6]{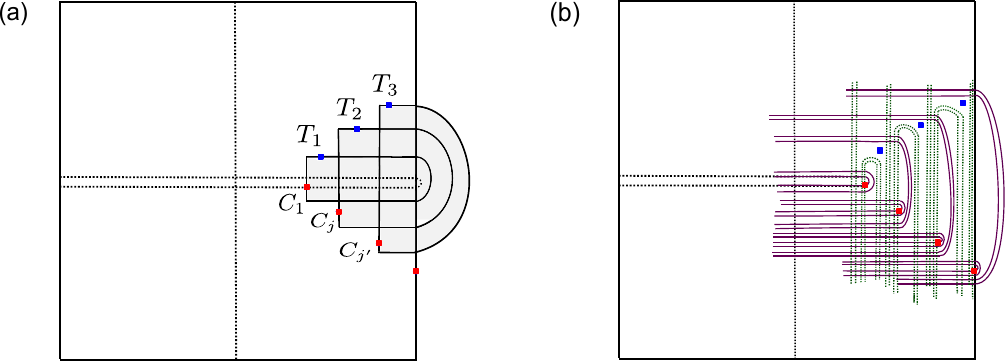}
\caption{Pruning region $\mathcal{P}_w$ associated to a concatenation of NBT orbits 
and the invariant manifolds of its pruning diffeomorphism $\psi$.}
\label{fig:nbtorbit}
\end{figure}

Now the proof continues pruning $\psi_1$ by using the domain $D_j$ which is defined by $C_j$ and its successor
 $C_{j'}$, that is, $D_j$ is bounded $\theta^i_s\subset W^s(C_j)$ and $\theta^j_u\subset W^u(T_j)$ where 
 $T_j=(c_{q_j}0\cdots0c_{q_{j'}}1)^{\infty}$. Since the invariant manifolds of the points of $K_{\psi_1}$
are just deformations of the invariant manifolds of the Smale horseshoe (proposition \ref{prop:topology1})
it is sufficiente to prove that $T_j\in K_{\psi_1}$ because in that case it maintains its invariant 
manifolds and then $D_j$ is still a domain for $\psi_1$. 

 Since $\{q_i\}$ is a P-list, as it was proved for $T_1$, one can conclude that $\operatorname{Orb}(T_j)$ does not intersect
 $\operatorname{Int}(D_1)$ and $\operatorname{Int}(D_j)$. Hence $T_j\in K_{\psi_1}$ and $D_j$ is a pruning domain.
 Thus we can construct the pruning diffeomorphism $\psi_j$ associated to $D_j$. 

Proceeding in the same way with \textit{all} the limiting points we will 
arrive to a pruning diffeomorphism $\psi$ with a basic set $K_{\psi}$ and  without bigons relative to $P^w_0$.
Moreover $K_{\psi}=\{\bold{x}\in \Sigma_2: F^i(\bold{x})\notin\mathcal{P}_w,\forall i\in\mathbb{Z}\}$ up a finite-to-one 
semiconjugacy, where $\mathcal{P}_w=\cup_{C_i\in\mathbb{L}}\operatorname{Int}(D_i)$.
By theorem \ref{thm:first}, $\mathcal{C}_{P_0^w}=K_{\psi}$ up a finite-to-one semiconjugacy which is injective on the 
set of non-boundary periodic points. See figure \ref{fig:nbtorbit}.
\end{proof}

If $\{q_i\}$ is not a P-list then some iterate of some $T_i$ belongs to a pruning 
domain $D_k$. It implies that the unstable manifolds in $D_k$ are deformed before 
of making the pruning isotopy in $D_k$. So we have no more control on these invariant 
segments and then the proof above can not be implemented.

As in the case of maximal decorations, there exists a description of the boundary 
periodic points of $K_{\psi}$.
\begin{proposition}
Let $\psi$ be the pruning map associated to  a P-list $\{q_i\}$ and 
let $C_i$ be a limiting point whose successor is $C_j$. Then the periodic orbit $B$ with code
 $B_{i,j}1B_{i,j}0B_{i,j}1$, where $B_{i,j}=c_{q_j}0\cdots0c_{q_{j-1}}$, is an $u$-boundary
 periodic orbit of $\psi$.  The unstable manifolds of 
 $(B_{i,j}1B_{i,j}0B_{i,j}1)^{\infty}$, $(B_{i,j}0B_{i,j}1B_{i,j}1)^{\infty}$ and $(B_{i,j}1B_{i,j}1B_{i,j}0)^{\infty}$ define a $3$ sides principal region that contains a $3$-pronged periodic 
 singularity with code $B_{i,j}1$.
\end{proposition}

\begin{example} 
Consider the homoclinic orbit $P^w_0$ defined by 
the P-list $\{2/5,2/7,1/3\}$ with code
${}^\infty010c_{2/5}0c_{2/7}0c_{1/3}0\cdot10{}^\infty={}^\infty010101101010011001010010\cdot10{}^\infty.$ By theorem \ref{thm:nbthomo}, its pruning region is formed by the union of the 
interior of the following two pruning domains:
$D_1$ bounded by a segment of the stable manifold of the limiting point
 $C_1={}^\infty010\cdot c_{2/5}0c_{2/7}0c_{1/3}010{}^\infty$ and a segment of the unstable 
 manifold of $T_1=(c_{2/5}0c_{2/7}1)^\infty=(1011010100110011)^\infty$, and $D_2$ bounded 
 by a stable segment of the limiting point $C_3={}^\infty010c_{2/5}0c_{2/7}0\cdot c_{1/3}010{}^\infty$
  and an unstable segment of $T_2=(c_{1/3}1)^\infty=(10011)^\infty$.
These domains are represented in figure \ref{fig:example1}.
\begin{figure}[h]
\centering
\includegraphics[width=135mm,scale=0.7]{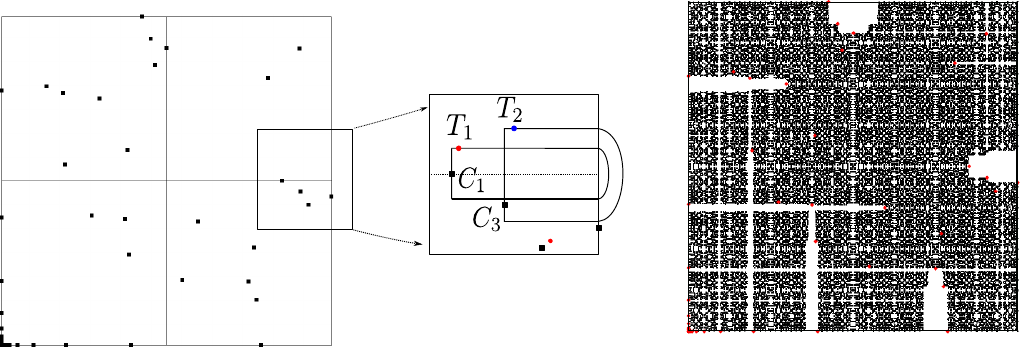}
\caption{Pruning region $\mathcal{P}_w$ associated to the 
homoclinic orbit $P^w_0={}^\infty010c_{2/5}0c_{2/7}0c_{1/3}0\cdot10{}^\infty={}^\infty010101101010011001010010\cdot10{}^\infty$, and the set of orbits forced by it.}
\label{fig:example1}
\end{figure}
\end{example}
A direct consequence of theorem \ref{thm:nbthomo} is a relation between decorations with 
the same \textit{combinatorics}.
\begin{definition}
Two P-lists $\{q_i\}_{i=1}^n$ and $\{q'_i\}_{i=1}^n$ have the same combinatorics if 
$q_i<q_j\Longleftrightarrow q'_i<q'_j$.
\end{definition}
So we can prove the following.
\begin{corollary}
Let $\{q_i\}_{i=1}^n$ and $\{q'_i\}_{i=1}^n$ be two P-lists with the same combinatorics.
If $q_i<q'_i$ for all $i=1,\cdots, n$ then $$P^{c_{q_1}0\cdots0c_{q_n}}_0\geqslant_2P^{c_{q'_1}0\cdots0c_{q'_n}}_0.$$
\end{corollary}
\begin{proof}
Let $\{D_i\}$ and $\{D'_i\}$ be the pruning domains associated to $P^{c_{q_1}0\cdots0c_{q_n}}$ and 
  $P^{c_{q'_1}0\cdots0c_{q'_n}}$, respectively. Since 
$\{q_i\}$ and $\{q'_i\}$ have the same combinatorics then there exists the same number
of pruning domains for each homoclinic orbit. By  hypothesis $q_i<q'_i$ which 
implies that $c_{q_i}0>_1 c_{q'_i}0$. By definition of $D_i$ and $D'_i$ it follows
that $D_i\subset D'_{i}$ and so  $\operatorname{Orb}(P^{c_{q'_1}0\cdots0c_{q'_n}}_0)\cap \mathcal{P}=\emptyset$, 
where $\mathcal{P}$ is the pruning region of $P^{c_{q_1}0\cdots0c_{q_n}}_0$. 
So the conclusion follows from theorem \ref{thm:nbthomo}.
\end{proof}

\subsection{Star homoclinic orbits}\label{sec:star}
In \cite{TanYam3} Yamaguchi and Tanikawa have dealt \textit{star} homoclinic orbits
$P^q_0$ which have as codes ${}^\infty0\cdot c_{q}0{}^\infty={}^\infty0\cdot10^{\mu_1}1^20^{\mu_2}1^2\cdots1^20^{\mu_{m-1}}1^20^{\mu_m}10{}^\infty$, 
with $q\in\widehat{\mathbb{Q}}$. These orbits have received this name because their
 train track types are star \cite{dCarHallBra,KinCol}. Here we will see that they have 
  well-defined pruning regions. To define it, construct the domain $D_q$ bounded by  
  a segment $\theta_s\subset W^s(\sigma^2({}^\infty0\cdot c_{q}0{}^\infty))$
and a segment $\theta_u\subset W^u(1^\infty)$ which intersect at the heteroclinic
points ${}^{\infty}1\cdot 0^{\mu_1-1}1^20^{\mu_2}1^2\cdots1^20^{\mu_{m-1}}1^20^{\mu_m}10{}^\infty$
and ${}^{\infty}10\cdot 0^{\mu_1-1}1^20^{\mu_2}1^2\cdots1^20^{\mu_{m-1}}1^20^{\mu_m}10{}^\infty$. 
See figure \ref{fig:homostar}(a).
Using the properties of $c_q$, we can prove that $\mathcal{P}_w=\operatorname{Int}(D_q)$ is a pruning 
domain associated to $P_0^q$ and  that its pruning diffeomorphism does not have bigons relative to $P^q_0$.  
Hence we have the following result.

\begin{proposition}
$\mathcal{C}_{P^q_0}=\{\bold{x}\in \Sigma_2:\operatorname{Orb}(\bold{x},F)\cap\operatorname{Int}(D_q)=\emptyset\}$ up to a finite-to-one semi-conjugacy which is injective on the set of non boundary periodic points.
\end{proposition}

Noting that $q\geq q'$ if and only if $D_q\subset D_{q'}$, one have the 
following consequence of proposition above. 

\begin{proposition}{\cite[Theorem 5.2.1]{TanYam3}}
For star homoclinic orbits, the forcing order agrees with the order 
 of the rational numbers, that is, $q\geq q'\Longleftrightarrow P^{q}_0\geqslant_2 P_0^{q'}$.
 \end{proposition}

As in proposition \ref{prop:pronged}, we can prove that the basic set of the pruning diffeomorphism $\psi$ 
associated to $D_q$ has a boundary periodic orbit $R_q$ of period $n$ whose code is $(\widetilde{c_q}1)^\infty$, 
where $\widetilde{c_q}$ is $c_q$ dropping the last two symbols. That  orbit $R_q$ 
defines a principal region of $n$ sides which contains the fixed point $1^\infty$.

\begin{example}
Consider $q=2/7$. Then $P^{2/7}_0$ has as code 
${}^\infty0\cdot c_{2/7}0{}^\infty={}^\infty0\cdot100110010{}^\infty$.
Its pruning region $D_{2/7}$ is defined by a stable segment of 
$\sigma^2({}^\infty0\cdot c_{2/7}0{}^\infty)={}^\infty010\cdot0110010{}^\infty$
and an unstable segment of the fixed point $1^\infty$, and it is represented in 
figure \ref{fig:homostar}(a) until its fourth iterate.
In figure \ref{fig:homostar}(b) we have represented the periodic orbits with 
periods less than $17$ which  are forced by $P^{2/7}_0$.
 \begin{figure}[h]
\centering
\includegraphics[width=115mm,scale=0.6]{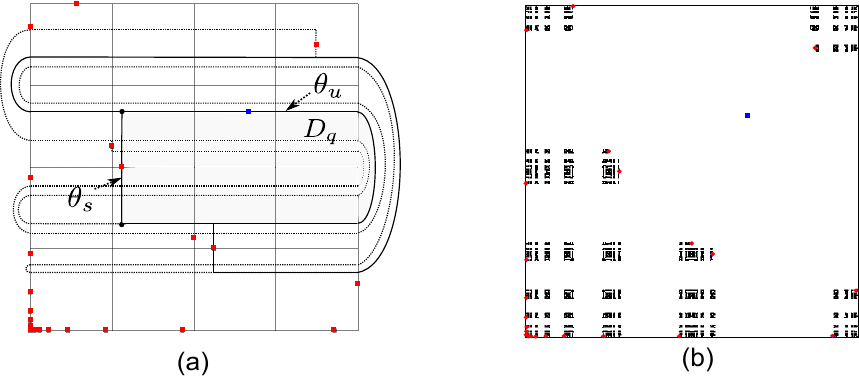}
\caption{(a) Pruning region $D_q$ of the homoclinic orbit 
$P^{2/7}_0={}^\infty0\cdot100110010{}^\infty$. 
(b) The set of orbits of periods less than $17$ which are forced by $P^{2/7}_0$.}
\label{fig:homostar}
\end{figure}
 \end{example}

\begin{example}
As an example of application to a concrete system, here we show a 
H\'enon map with a well-defined pruning region. Consider the H\'enon map
$H_{0,1.11}(x,y)=(-x^2-1.11y,x)$. Numerically the structure of the invariant manifolds
of this map  is similar to the structure of the 
pruning map associated to the star homoclinic orbit $P^{1/5}_0=P_0^{00}={}^\infty010000\cdot10^\infty$. 
See figure \ref{fig:homo}(a).
The pruning domain associated to that
orbit is the domain $D_{1/5}$ depicted in figure \ref{fig:homo}(b) which is bounded
 by $\theta_s\subset W^s({}^\infty010\cdot00010^\infty)$ and 
$\theta_u\subset  W^u(1^\infty)$. Hence the orbits of 
$H_{0,1.11}$, represented in figure \ref{subfighomoc}, correspond to 
the points forced by $P_0^{00}$, 
$\Sigma_{P^{00}_0}=\{\bold{x}\in \Sigma_2:\operatorname{Orb}(\bold{x},F)\cap\operatorname{Int}(D_{1/5})=\emptyset\}$  up a finite-to-one semi-conjugacy.
\begin{figure}[h]
\centering
\subfigure[H\'enon map $H_{0,1.11}(x,y)$]{
\includegraphics[width=43mm,scale=0.6]{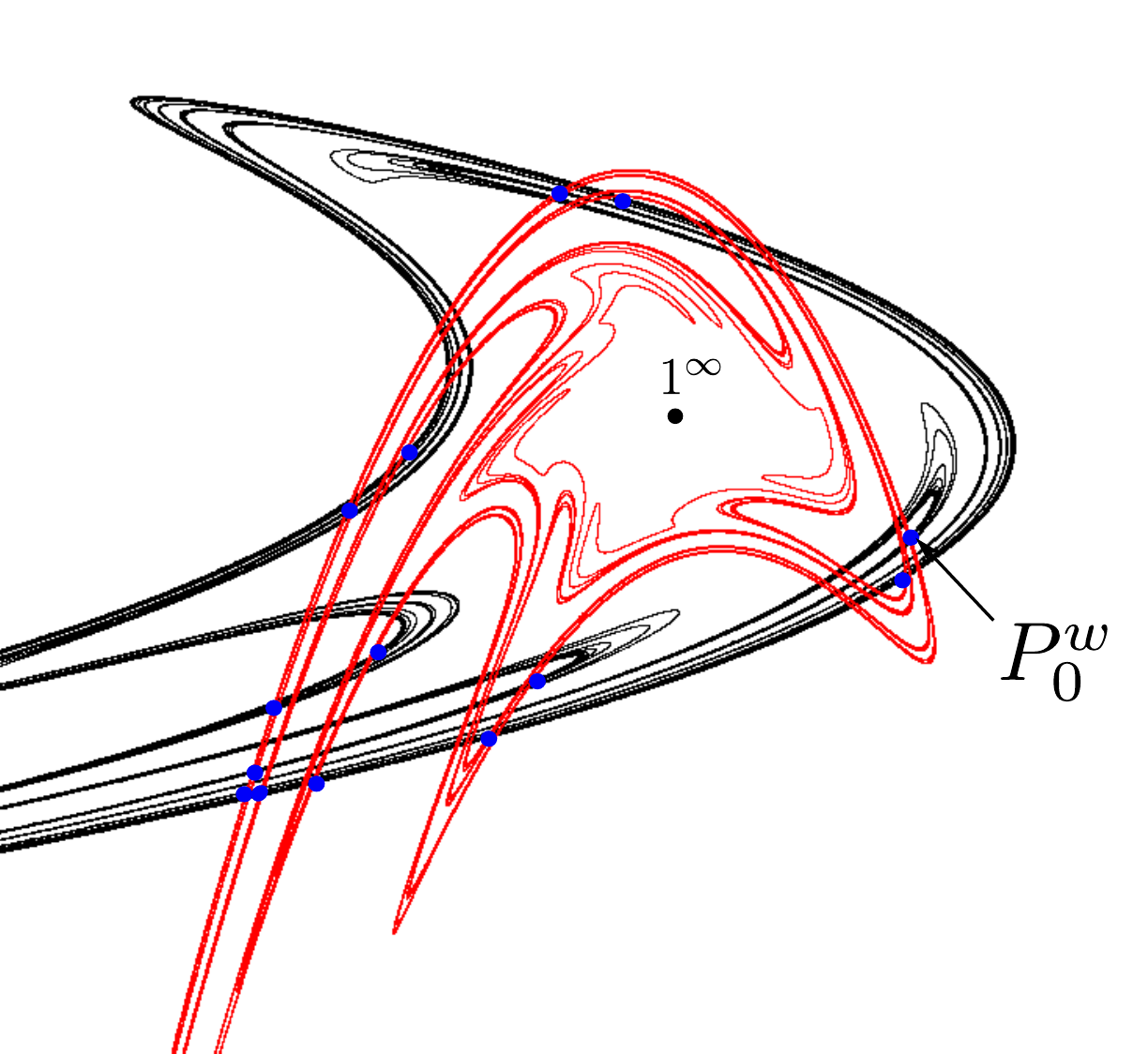}
\label{subfighomoa}
}
\hspace{5pt}
\subfigure[Pruning domain]{
\includegraphics[width=44.5mm,scale=0.6]{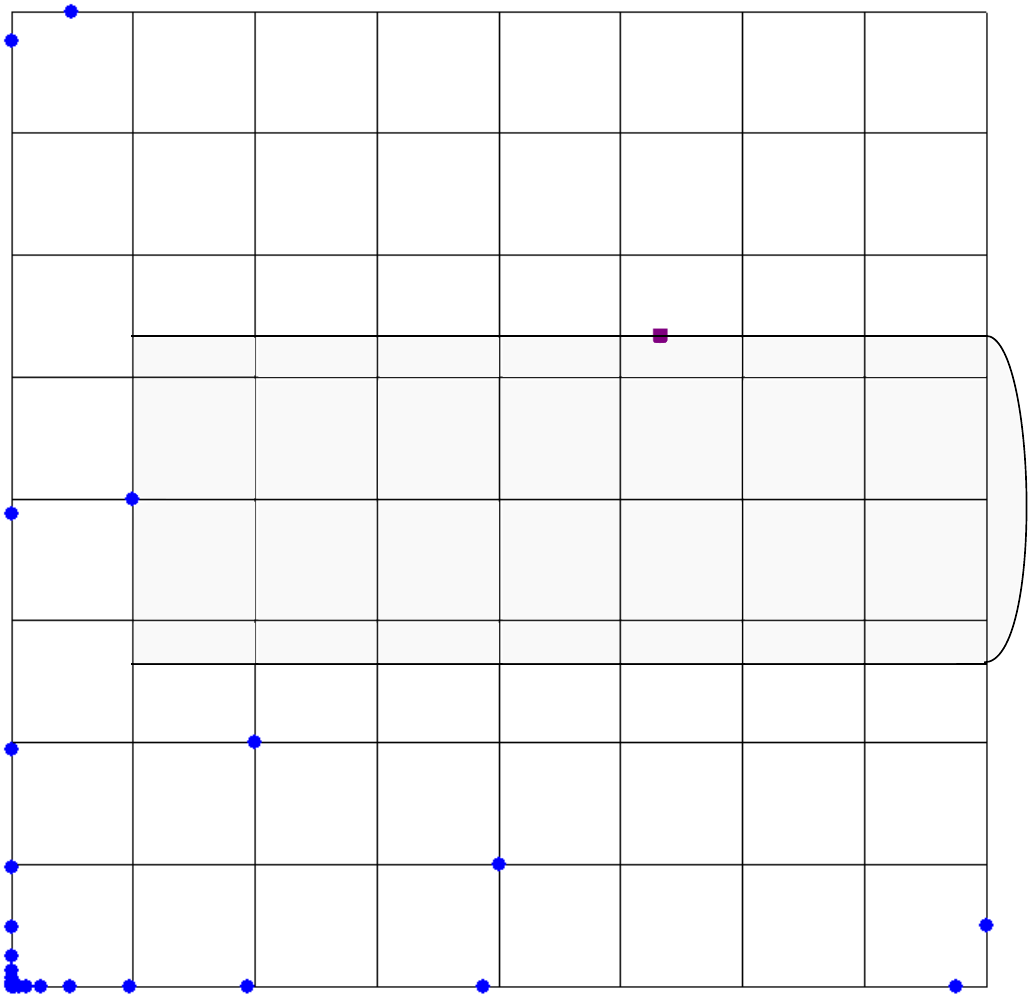}
\label{subfighomob}
}
\hspace{5pt}
\subfigure[Orbits forced by $P^{00}_0$]{
\includegraphics[width=43mm,scale=0.6]{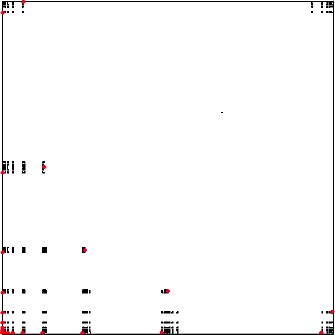}
\label{subfighomoc}
}
\hspace{5pt}
\caption{\small{Homoclinic orbit $P_0^{00}={}^\infty010000.10^\infty$, its pruning region
in the symbol square and the set of orbits  forced by $P^{00}_0$.}}
\label{fig:homo}
\end{figure}
\end{example}

 In the examples showed in previous subsections, there exists a well-defined 
 pruning region obtained applying the differentiable pruning theorem a finite number of times.
 It is not true for all decorations, in fact, there exist decorations $w$ for which
 their pruning regions are formed by an infinite number of pruning domains. But, even in 
that case, our method says that $\mathcal{C}_{P^w_0}$ is the set of orbits that do not 
intersect those pruning domains except for a certain type of boundary periodic orbits.
 See \cite{MenOrd} for examples and details.

\section{A method for eliminating bigons of a Smale map}\label{sec:pruningbigons}
In the following suppose that $f$ is an  orientation preserving Smale map on $D^2$ 
having a saddle set $K$. From section \ref{background}, to solve the problem of 
finding  orbits forced by  a homoclinic orbit $P$, we just have to 
find a Smale map without bigons relative to $P$.
In this section we will introduce the \textit{pruning method}  that will 
allow us to eliminate the bigons of an arbitrary Smale map relative to a 
homoclinic orbit. This method generalizes the applications, given in section 
\ref{sec:horseshoe}, of the differentiable pruning theorem on homoclinic 
orbits of the Smale horseshoe.

We will prove the following theorem.
\begin{theorem}\label{thm:methodthm}
 Let $f$ be a Smale diffeomorphism  having a wandering or non-wandering 
bigon $\mathcal{I}$  and let $P\subset K$ be an orbit (periodic or homoclinic) 
disjoint from $\mathcal{I}$. 
Then there exists a pruning domain containing  $\mathcal{I}$ and disjoint from $P$. 
\end{theorem}
Next subsections are devoted to the proof of theorem \ref{thm:methodthm}. 
We will prove it only for wandering bigons, since minor modifications can 
state the same conclusion for non-wandering bigons. 
Now we will give the details.
\subsection{Finding the maximal domain}\label{sec:maximal}
Let $\mathcal{I}$ be a wandering bigon with $\partial\mathcal{I}=\alpha_s\cup \alpha_u$
and let $P$ be a periodic or a homoclinic orbit to a fixed point $p$.  
This section describes how to find the \textit{maximal domain}, relative to $P$, 
which contains $\mathcal{I}$. The ideia consists in finding the 
maximal region $D$ which is a region obtained enlarging the bigon 
$\mathcal{I}$ to a bigger non-wandering bigon. Such region, which is bounded by a 
stable segment $\theta_s$ and an unstable segment $\theta_u$,  contains 
$\mathcal{I}$ and a rectangle $\mathcal{R}$, that is, a region for which 
there exists a homeomorphism $h:[0,1]\times[0,1]\rightarrow \mathcal{R}$ such that:
\begin{itemize}
\item $\mathcal{R}\cap W^s(K)=h(F_s \times [0,1])$, where $F_s$ is a closed
subset of $[0,1]$ and $\{0,1\}\subset F_s$,
\item $\mathcal{R}\cap W^u(K)=h([0,1]\times F_u)$, where $F_u$ is a closed
subset of $[0,1]$ and $\{0,1\}\subset F_u$.
\item $D\cap K=\mathcal{R}\cap K$.
\end{itemize}
Note that there are two possibilities for $D$: (a) $D$ is bounded by $\theta_s\subset  W^s(p_s)$ and 
$\theta_u\subset W^u(p_u)$ where $p_s$ and $p_u$ are $s$-boundary and $u$-boundary periodic
points, respectively and (b) $\theta_{s}\subset W^s(r)$ and $\theta_u\subset W^u(r)$ where
$r$ is both $s$-boundary and $u$-boundary periodic point. See figure \ref{fig:maxdomain}.
\begin{figure}[h]
\centering
\includegraphics[width=115mm,height=55mm,scale=0.6]{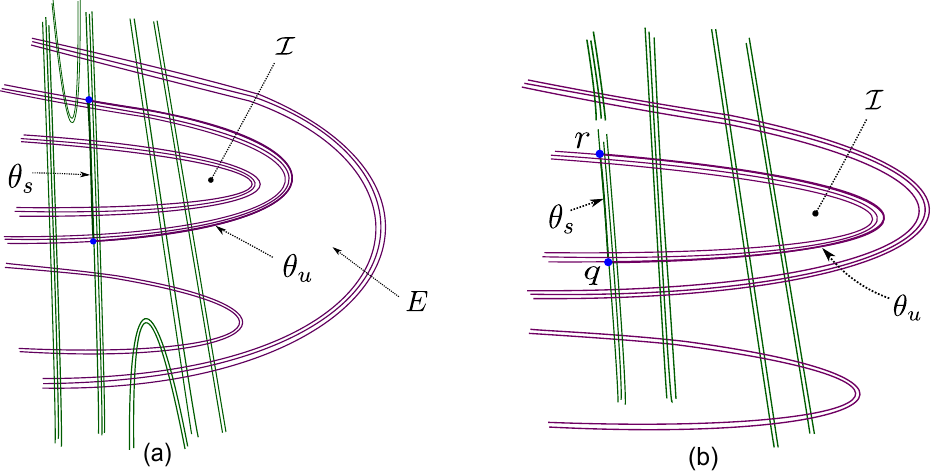}
\caption{The maximal domain $D$ containing $\mathcal{I}$ is bounded 
by the invariant segments $\theta_u$ and $\theta_s$.}
\label{fig:maxdomain}
\end{figure}

If an element $p_j$ of the orbit $P$ lies  in $\operatorname{Int}(D)$ then
we consider  the domain $D_j$ which is bounded by a stable segment $\theta_s\subset W^s(p_j)$ 
(in the case $P$ is periodic) or $\theta_s\subset W^s(p)$ ( in the case that $P$ is homoclinic)
which contains $p_j$. See figure \ref{fig:maxdomainp}. Thus the maximal domain $D$ is the smallest
of all these domain $D_j$.
 \begin{figure}[h]
\centering
\includegraphics[width=115mm,scale=0.6]{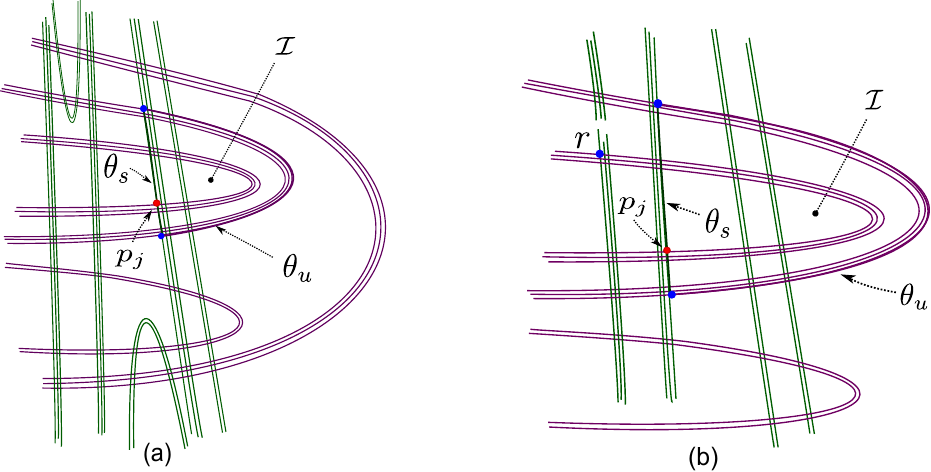}
\caption{The maximal domain $D$ containing $\mathcal{I}$ is bounded 
by the invariant segments $\theta_u$ and $\theta_s$.}
\label{fig:maxdomainp}
\end{figure}
In all the cases or a piece of $\theta_u$ bounds a region $E$ that is 
wandering and whose boundary is formed by $n$ stable segments and $n$ unstable 
segments with $n\ge3$, or $\theta_u$ contains a boundary period point. See figure 
\ref{fig:maxdomain}.

If the maximal domain is a pruning domain then theorem \ref{thm:methodthm} is 
proved. Thus it remains to consider the case when $D$ is not a pruning domain.
\subsection{Defining a pruning domain from $D$}
 If the maximal region $D$ is not  a pruning domain one must prove that it 
is always possible to decrease $\theta_s$ until a sub-segment 
$\theta'_s\subset\theta_s$ which defines a pruning domain $D'$, that 
contains $\mathcal{I}$,
whose boundary consists of $\theta'_s$ and an unstable segment $\theta'_u$. 

We need the following lemma.
\begin{lemma}\label{lem13}
If  $f^n(\theta_s)\cap \operatorname{Int}(D)=\emptyset$, for all $n\ge1$ then
$D$ satisfies pruning  conditions for $f$.
\end{lemma}
\begin{proof}
By hypothesis $f^n(\mathring{\theta_s})\cap\theta_u=\emptyset$ for
 all $n\ge1$ then $f^{-n}(\theta_u)\cap\mathring{\theta_s}=\emptyset$
  for all $n\ge1$. So if there exists $n\ge 1$ such that $f^{-n}(\theta_u)\cap\operatorname{Int}(D)\neq\emptyset$ then $f^{-n}(\theta_u)\subset\operatorname{Int}(D)$. This is not possible 
since a piece of $\theta_u$  is in the boundary of $E$ or $\theta_u$ contains 
a boundary periodic point $r$ and, then $f^{-n}(E)\subset\operatorname{Int}(D)$ or $f^{-n}(r)\subset\operatorname{Int}(D)$, which is a contradiction with the definition of $D$. 
So $D$ satisfies the pruning conditions of
 definition \ref{def:pruningconditions}.
\end{proof}

By lemma \ref{lem13} we just need to study the case when
there exist positive integers $N_i$ such that
 $f^{N_i}(\theta_s)\cap \operatorname{Int}(D)\neq\emptyset$.

\begin{lemma}\label{prop:inter1}
Let $D$ be the maximal domain and an integer $N>1$. If 
$f^{N}(\theta_s)\cap\operatorname{Int}(D)\neq\emptyset$
 then $f^{N}(\mathring{\theta_s})\cap \theta_u\neq\emptyset$ 
 or $\alpha_s\subset f^{N}(\theta_s)$.
\end{lemma}
\begin{proof}
Let $\gamma_s$ be the stable segment such that $\alpha_s\subset\gamma_s$ 
and the end-points of $\gamma_s$ are included in $\theta_u$.
If $f^{N}(\mathring{\theta_s})\subset(W^s(K)\cap D)\setminus\gamma_s$, then
$f^{N}(\theta_s)$ can be continued through the leaves of $W^u(K)\cap D$.
 Finding the $N$-backward iterate of these continuation, we can see
  that the rectangle defining $D$ can be extended to a bigger one. That is a 
contradiction with the definition of $D$.  The same argument works 
  if $f^{N}(\theta_s)\subset\gamma_s\setminus\alpha_s$.
The remaining case is when $f^{N}(\theta_s)\cap\mathcal{I}\neq\emptyset$ 
that implies $\alpha_s\subset f^{N}(\theta_s)$.

If $f^N(\mathring{\theta_s})$ is not included in $W^s(K)\cap D$ then 
$f^N(\mathring{\theta_s})\cap\theta_u\neq\emptyset$.
\end{proof}

\begin{figure}[h]
\centering
\includegraphics[width=80mm,scale=0.6]{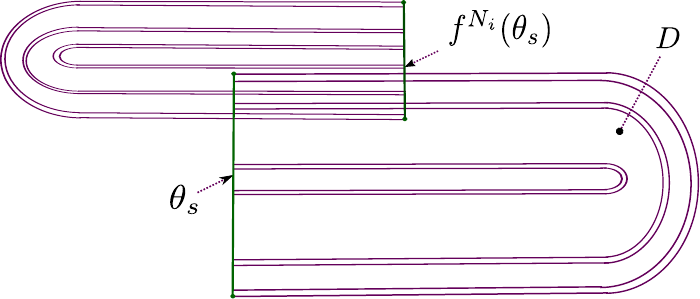}
\caption{If the maximal domain is not a pruning domain, an iterate $f^{N_i}(\theta_s)$ of $\theta_s$
intersects $\operatorname{Int}(D)$ as in the Figure.}
\label{fig:intersection}
\end{figure}
Actually there is only a finite number of positive integers satisfying conditions of lemma above.
\begin{lemma}\label{lem:finite}
There are a finite number of positive integers 
$N_1,N_2,...,N_l$ such that  $f^{N_i}(\theta_s)\cap\operatorname{Int}(D)\neq\emptyset$, 
for all $i=1,...,l$. 
\end{lemma}
\begin{proof}
Suppose that there exists a sequence $\{N_i\}_{i\in\mathbb{N}}$ of positive integers 
such that  $f^{N_i}(\theta_s)\cap\operatorname{Int}(D)\neq\emptyset$. 
See figure \ref{fig:intersection}. 
 By lemma \ref{prop:inter1} and since $\operatorname{diam}(f^{N_i}(\theta_s))$ goes to
 $0$ when $i$ goes to $\infty$, it follows that only a finite number of the $N_i$
 satisfy $f^{N_i}(\theta_s)\supset\alpha_s$. So we just have to consider the 
case when  $f^{N_i}(\mathring{\theta_s})\cap\theta_u\neq\emptyset$. 
Since $\theta_s\subset W^s(p_s)$, there exist a subsequence $\{N'_k\}\subset\{N_i\}$ and 
a point $p_s^0$ of the orbit of $p_s$ such that $\lim f^{N'_i}(\theta_s)$ goes to $p_s^0$
when $i$ goes to $\infty$, and then one have that $p_s^0\in\theta_u$. Hence $p_s^0$ is a periodic
 point in $W^u(p_u)$ which  is only possible if $p_s^0=p_u\in \theta_s\cap\theta_u$ and $p_s^0=p_s$. 
 But in this case  $f^{n}(\theta_s)\cap\operatorname{Int}(D)\neq\emptyset$ 
  for a finite number of  positive integers. That is a contradiction.
\end{proof}

For the iterates of lemma \ref{lem:finite}, we have:
\begin{lemma}\label{prop:4}
For every $N_i$ there exists a subdomain $D_i\subset D$ such that
 $\partial D_i=\theta_{s,i}\cup\theta_{u,i}$ where 
 $\theta_{s,i}\subset\theta_s$, $f^{N_i}(\theta_{s,i})\cap\operatorname{Int}(D_i)=\emptyset$
and $\theta_{u,i}$ is an unstable leaf included in $D$.
\end{lemma}
\begin{proof}
Since $f^{N_i}(\theta_s)\cap\operatorname{Int}(D)\neq\emptyset$ one see that
the projection of $f^{N_i}(\theta_s)$ along the unstable leaves of $D$ intersects
$\theta_s$. Thus one can decrease $\theta_s$ until a small segment $\theta_{s,i}$ such that
(a) $\theta_{s,i}$ has its end-points in the same unstable leaf $\theta_{u,i}$ and 
(b) $f^{N_i}(\theta_{s,i})$ projects along the unstable leaves and  intersects 
$\theta_{s,i}$ at only one point. See figure \ref{fig:subdomain}.
\end{proof}
\begin{figure}[h]
\centering
\includegraphics[width=85mm,scale=0.6]{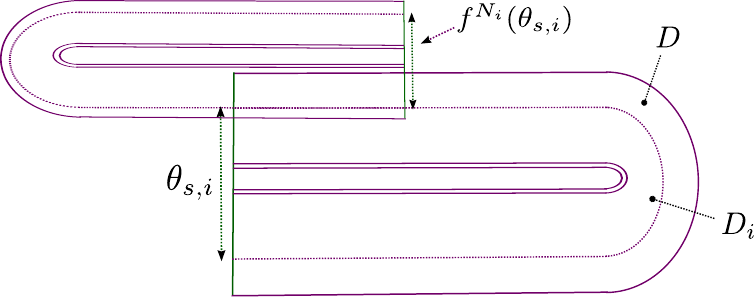}
\caption{Decreasing the domain $D$ to another $D_i$.}
\label{fig:subdomain}
\end{figure}

With the preceding lemmas we can prove now theorem \ref{thm:methodthm}.
\begin{proof}[Proof of theorem \ref{thm:methodthm}]
If the maximal domain $D$ of section \ref{sec:maximal} satisfies the 
pruning conditions then theorem \ref{thm:methodthm} is proved.
If $D$ does not satisfy pruning conditions then by, lemma \ref{lem:finite},  $f^N(\theta_s)\cap\operatorname{Int}(D)=\emptyset$ for all positive integer $N\ge1$ except 
for a finite number of positive integers $N_i$ with $i=1,\cdots,l$. 
From lemma \ref{prop:4} for each $N_i$ there exists
a subdomain $D_i\subset D$ with boundary $\theta_{s,i}\cup\theta_{u,i}$ such that
$\theta_{s,i}\subset\theta_s$ and $\theta_{u,i}$ is an unstable segment joining 
the end-points of $\theta_{s,i}$.

Since $\theta_{s,i}\subset\theta_{s,j}$ or $\theta_{s,j}\subset\theta_{s,i}$, 
for all $i,j\in\{1,..,l\}$, the domains $\{D_i\}_{i=1}^l$ can be organized by 
inclusion.  Let $i'$ be the positive integer such that $\theta_s':=\theta_{s,i'}$ is the smallest segment in $\{\theta_{s,i}\}$, so the domain $D':=D_{i'}$ is the smallest of all
 the domains $D_1,\cdots,D_l$. Let $\theta_{u}'$ be the unstable segment joining the end-points  of $\theta'_{s}$. So $\partial D'=\theta_s'\cup \theta_u'$ and 
 $f^{N_i}(\theta'_s)\cap\operatorname{Int}(D')=\emptyset$, for all $i=1,\cdots,l$.
Hence
\begin{equation}\label{eq:stable}
 f^n(\theta'_s)\cap\operatorname{Int}(D')=\emptyset, \textrm{  for all } n\ge1.
\end{equation}
Therefore for proving that $D'$ is a pruning domain it is sufficient to prove that 
 \begin{equation}\label{eqbackward}
f^{-n}(\theta'_u)\cap\operatorname{Int}(D')=\emptyset, \textrm{ for all }n\ge1.
 \end{equation}
  By construction of $D'$, it follows that $f^{N_{i'}}(\theta'_u)\cap \theta'_{u}\neq\emptyset$
and so by \cite[Theorem 1.2]{Ply}, $\theta_u'$ is included in the unstable
  manifold of a periodic point $q_u$. 
	\begin{figure}[h]
\centering
\includegraphics[width=100mm,scale=0.6]{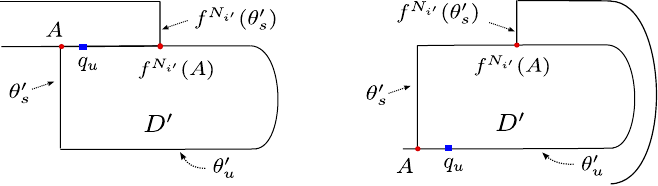}
\caption{The unstable piece $\theta'_u$ contains a periodic point $q_u$.}
\label{fig:periodic}
\end{figure}
If $A$ is the end-point of $\theta'_s$ such that $f^{N_{i'}}(A)\in\theta'_u$ then, 
since $f$ preserves orientation, one have that $f^{N_{i'}}(\theta'_u)$ contains $A$. 
See figure \ref{fig:periodic}. Decreasing $\theta'_u$ it follows that there 
exists a periodic point in $\theta'_u$ which, by uniqueness, has to be $q_u$.
Hence $q_u\in\theta'_u$.

Now we will prove that $D'$ satisfies condition (\ref{eqbackward}).
Since  $f^n(\mathring{\theta'_s})\cap\theta'_u=\emptyset$, for all $n\ge1$, then
 $\mathring{\theta'_s}\cap f^{-n}(\theta'_u)=\emptyset$, for all $n\ge1$. It implies that  
 if $f^{-n}(\theta'_u)\cap\operatorname{Int}(D')\neq\emptyset$ then 
 $f^{-n}(\theta'_u)\subset\operatorname{Int}(D')$.
Suppose that there are positive integers $N$ such that $f^{-N}(\theta'_u)\subset\operatorname{Int}(D')$. 
Since $\theta'_u$ has a periodic point $q_u$ of period $N_{i'}$ and $f^{-1}$ is contracting on 
$\theta'_u$, there is a finite number
 of positive integers $M_{i_j}$ with $M_{i_j}<N_{i'}$, $j=1,..,k$ such that
  $f^{-M_{i_j}}(\theta'_u)\subset\operatorname{Int}(D')$. Recall that $f^{N_{i'}}(\theta'_s)$ 
  intersects $\theta'_u$ at the point $f^{N_{i'}}(A)$. 
  So $f^{-M_{i_j}+N_{i'}}(A)\in f^{-M_{i_j}+N_{i'}}(\theta'_s)\cap \operatorname{Int}(D')$. It is a contradiction
  with (\ref{eq:stable}) since $-M_{i_j}+N_{i'}>0$.	
	Therefore the domain $D'$, associated to the bigon $\mathcal{I}$, satisfies
 pruning conditions which proves theorem \ref{thm:methodthm}.
 \end{proof}

\subsection{Pruning Method} 
 Now it is clear how the pruning method works.
 \begin{algorithm} Given an orbit (periodic or homoclinic) $P$ of a Smale map $f$ 
 then:
  \begin{itemize}
 \item[(1)] If $f$ has a bigon relative to $P$, applying  theorem \ref{thm:methodthm} 
 to $f$ and $P$, we can find a pruning domain $D_1$.
 Then theorem \ref{theo:pruning} allows us to construct a pruning diffeomorphism $f_1$. 
 \item[(2)] If $f_1$ has no bigons relative to $P$ then, by theorem \ref{thm:first}, 
every periodic orbit  of $f_1$  is forced  by $P$.
  \item[(3)] If $f_1$ has bigons then return to (1).
  \item[(4)] Repeating (1), (2) and (3) we will obtain a sequence $f_1,f_2\cdots$.
  If this sequence is finite $\{f_1,\cdots, f_n\}$  then, by   theorem \ref{thm:first}, the  orbits of 
	the basic set of $f_n$ are forced by $P$ up a finite-to-one semiconjugacy which is injective 
	on the set of periodic orbits.
 \end{itemize}
\end{algorithm}

\section{Acknowledgements}\label{ackref}
This research was partially supported by FAPESP grant 2010/20159-6.
I would like to thank the institute IME-USP for the hospitality during
the time in which parts of this work was done.
I am grateful to Andr\'e de Carvalho for his stimulating conversations about pruning, 
to Toby Hall that has written the useful programs \textit{Prune} and \textit{Train} 
to explore the pruning regions and the train track of periodic horseshoe points, 
and to Philip Boyland  who talk us about isotopy stability during his stay at IME-USP.
I would also thank the referee of the first draft of this manuscript who 
made many helpful comments for improving some unclear aspects
contained in the proofs.

\end{document}